\documentclass{article}%
\usepackage{graphicx}
\usepackage{amsmath,amssymb,amsfonts}%
\usepackage{theorem}%
\usepackage{color}%
\usepackage{hyperref}
\usepackage[font={small, it}]{caption}
\usepackage{caption}
\usepackage{subcaption}
\usepackage{tikz}

\theoremstyle{change}%
\newtheorem{definition}{Definition:}[section]%
\newtheorem{proposition}[definition]{Proposition:}%
\newtheorem{theorem}[definition]{Theorem:}%
\newtheorem{lemma}[definition]{Lemma:}%
{\theorembodyfont{\rmfamily}\newtheorem{remark}[definition]{Remark:}}%
{\theorembodyfont{\rmfamily}\newtheorem{example}[definition]{Example:}}%

\newenvironment{proof}
{{\bf Proof:}}
{\qquad \hspace*{\fill} $\Box$}%

\newcommand{\fg}{\mathfrak{g}}%
\newcommand{\fn}{\mathfrak{n}}%
\newcommand{\tr}{\operatorname{tr}}%
\newcommand{\id}{\operatorname{id}}
\newcommand{\rme}{\mathrm{e}}%

\newcommand{\CC}{\mathcal{C}}%
\newcommand{\EC}{\mathcal{E}}%
\newcommand{\ZC}{\mathcal{Z}}%
\newcommand{\AC}{\mathcal{A}}%
\newcommand{\XC}{\mathcal{X}}%
\newcommand{\DC}{\mathcal{D}}%
\newcommand{\R}{\mathbb{R}}%

\begin{document}

	\title{Isometries of Almost-Riemannian structures on nonnilpotent, solvable 3D Lie groups}%
	\author{V\'{\i}ctor Ayala%
		\thanks{
			Supported by Proyecto Fondecyt n$%
			{{}^\circ}%
			$ 1190142. Conicyt, Chile.} \\
		Instituto de Alta Investigaci\'{o}n\\
		Universidad de Tarapac\'{a}, Arica, Chile \and
Adriano Da Silva\thanks{Supported by Proyecto UTA Mayor Nº 4768-23}\\
		Departamento de Matem\'atica,\\Universidad de Tarapac\'a - Iquique, Chile.
		\and
 Danilo A. Garc\'{\i}a Hern\'andez\thanks{ Supported by Capes Finance Code 001}\\
		Instituto de Matem\'{a}tica\\
		Universidade Estadual de Campinas, Brazil\\
	}
	\date{\today }
	\maketitle
	
	\begin{abstract}
		 In this paper we prove that automorphisms are the only isometries between rank two Almost-Riemannian Structures on the class of nonnilpotent, solvable, connected 3D Lie groups. As a consequence, a classification result for rank two ARSs on the groups in question is obtained.
	\end{abstract}
	
	
	
	\section{Introduction}



 Almost-Riemannian structures (shortly ARS) are a part of sub-Riemannian geometry with nonconstant rank. They can be seen, roughly speaking, as a smooth orthonormal frame degenerating on a singular set, called the singular locus.
 They appear naturally as limits of Riemannian metrics and in the study of hypoelliptic operators \cite{Gru}, but other works highlighting nice physical features related to these structures can also be found in the literatura (see for instance \cite{Boscain5}). The first work on the subject appeared in 1956 in the work of T. Takasu \cite{tak}. However, works on the subject, from a control theoretical point of view, have appeared recently \cite{Ag0, Ag1, Ag2, DSAyD2, PJAyGZ1, Bon1, Bon2, Boscain1, Boscain2, Boscain3, Boscain4, Boscain5, PJAyGZ}. Among them, the concept of simple ARS on connected Lie groups was introduced in \cite{PJAyGZ1} by considering $n-1$ left-invariant vector fields and one linear vector field having full rank at a nonempty subset of $G$ and satisfying the Lie algebra rank condition. The strong relation of these structures with the group symmetries allow one to obtain strong results for the singular locus (see  \cite{DSAyD2, PJAyGZ1}) and for the isometries between ARSs (see \cite{PJAyGZ}).

 Concerning isometries of ARSs, the authors showed \cite[Theorem 4]{PJAyGZ} that the only possible isometries between ARSs on a nilpotent Lie group are the automorphisms of the group. By means of a counter-example (see \cite[Section 3.4]{PJAyGZ}) they showed that such a remarkable fact is not true, in general, for nonnilpotent groups. In the paper at hand, we study isometries in the class of nonnilpotent, solvable 3D Lie groups. We show that for rank two ARSs (see Definition \ref{ranktwo}) it still holds that the only isometries are the automorphisms of the group.  In order to prove that, we transfer the problem to the algebra level, by working with the differential at the identity element. Since, by symmetry, Lie groups have all their information centered at the origin, such an approach is the most appropriate. However, since one cannot assume a priori smoothness of the differential of the isometry over the singular locus, this task turns out to be quite hard. The way to overcome the previous difficulty, is to use the smoothness of the isometry outside the singular locus, as assured by the Myers-Steenrod Theorem \cite[Theorem 8]{Myers}, in association with curves that cross the singular locus (see \cite[Section 4.1]{DSAyD2}).

    The paper is divided as follows: Section 2 is used to state the framework of the paper. Here we define the 3D Lie groups we are interested, their algebras, group of automorphisms, linear and invariant vector fields, and so on. In sequence, the concept of ARS and their isometries are (formally) defined. We also make a compilation of the main results concerning the singular locus. In Section 3 we state and prove our main result (Theorem \ref{main}) which assures that, any isometry between rank two ARSs is an automorphism of the group. The proof is divided in three subsections. In the first one we prove our ``Fundamental Lemma'', namely, we show that if an isometry preserves one left-invariant vector field of the associated distribution, then it is in fact an automorphism. In the second subsection, we show that the differential of any rank two isometry keeps the nilradical of the algebra invariant. Since, by the Lie algebra rank condition, the intersection of the nilradical and the 2D left-invariant distribution of the ARS is one-dimensional, we can show that any isometry preserves a unitary left-invariant vector field generating this intersection and, by the fundamental lemma, it is an automorphism. As consequence of the main result, we are able to obtain a classification of rank two isometries of ARSs on the 3D groups under consideration.

	\section{Preliminaries}
	
	Here we introduce the main concepts and results we will need in the next sections. In order to maintain our work self-contained, all the concepts here will be done over the manifolds we are interested, making remarks when necessary to general cases.

	\subsection{Solvable nonnilpotent 3D Lie groups and algebras}
	
	According to \cite[Chapter 7]{Oni}, up to isomorphisms, any solvable, nonnilpotent 3D Lie algebra is given by the semi-direct product $\fg(\theta)=\R\times_{\theta}\R^2$, with Lie bracket determined by the relation  
	$$[(a, 0), (0, w)]=(0, a\theta w), \;\;\;(a, w)\in \R\times\R^2,$$
	where $\theta\in\mathfrak{gl}(2, \R)$ is a matrix of one of the following types:
	$$\left(\begin{array}{cc}
	1  & 1\\ 0 & 1 
	\end{array}\right), \;\;\;\;\left(\begin{array}{cc}
	1  & 0\\ 0 & \gamma 
	\end{array}\right), \gamma\in[-1, 1]\;\;\;\;\mbox{ or }\;\;\;\;\left(\begin{array}{cc}
	\gamma  & -1\\ 1 & \gamma 
	\end{array}\right), \gamma\in\R.$$

	The {\bf nilradical} $\fn(\theta)$ of $\fg(\theta)$ coincides with a (diffeomorphic) copy or $\R^2$ inside $\fg(\theta)$ and is given by $\fn(\theta)=\{0\}\times\R^2$. 

     At the group level, up to isomorphisms, the connected, simply connected 3D Lie group with Lie algebra $\fg(\theta)$ is $G(\theta)=\R\times_{\rho}\R^2$, where $\rho_t=\rme^{t\theta}$ and the product satisfies
    $$(t_1, v_1)(t_2, v_2)=(t_1+t_2, v_1+\rho_{t_1}v_2), \;\;\; (t_1, v_1), (t_2, v_2)\in\R\times_{\rho}\R^2.$$

Following \cite[Proposition 2.1 and 2.2]{DSAyD2} the algebra of {\bf derivations} of $\fg(\theta)$ and the groups of {\bf automorphisms} of $\fg(\theta)$ and $G(\theta)$ are given, respectively, by 
 \begin{equation}
     \label{der}
     \mathrm{Der}(\fg(\theta))=\left\{\left(\begin{array}{cc}
		0 & 0\\ \xi & A
		\end{array}\right), \xi\in\R^2, A\in\mathfrak{gl}(2, \R), \mbox{ with } A\theta=\theta A\right\},
 \end{equation}

 \begin{equation}  
 \mathrm{Aut}(\fg(\theta))=\left\{\left(\begin{array}{cc}
		\varepsilon & 0\\ \eta & P
		\end{array}\right), \eta\in\R^2, P\in\mathrm{Gl}(\R^2), \mbox{ with } P\theta=\varepsilon\theta P\right\},
  \end{equation}
  and, 
   \begin{equation}
       \label{auto}
       \mathrm{Aut}(G(\theta))=\left\{\phi(t, v)=\left(\varepsilon t, Pv+\varepsilon\Lambda_{\varepsilon t}^{\theta}\eta\right), \eta\in\R^2, P\in\mathrm{Gl}(\R^2), \mbox{ with } P\theta=\varepsilon\theta P\right\},
   \end{equation}
   
   		where $\varepsilon=1$ if $\tr\theta\neq 0$ or $\varepsilon\in\{-1, 1\}$ if $\tr\theta=0$. Here, for any matrix $A\in \mathfrak{gl}(2, \R)$ we have that
     $$\Lambda^A: (t, w)\in \R\times \R^2\mapsto \Lambda_t^Aw:=\int_0^t\rme^{sA}w ds.$$

   	A {\bf linear} and a {\bf left-invariant} vector field on $G(\theta)$ are given, respectively, as
    $$\XC(t, v):=(0, Av+\Lambda^{\theta}_t\xi)\hspace{.5cm}\mbox{ and }\hspace{.5cm}Y^L(t, v):=(\alpha, \rho_t\eta),$$
    where $\alpha\in\R$, $\xi, \eta\in\R^2$ and $A\in\mathfrak{gl}(2, \R)$ satisfies $A\theta=\theta A$. We denote by $\ZC_{\XC}$ the set of {\bf singularities} of $\XC$ (or the set of fixed points of $\varphi_s)$, that is, 
     $$\ZC_{\XC}=\{(t, v)\in G(\theta); \;\;\XC(t, v)=0\}.$$
     
     The flow of $\XC$ is a $1$-parameter subgroup of $\mathrm{Aut}(G(\theta))$ and is given by 
    $$\varphi_s(t, v)=\left(t, \rme^{sA}v+\Lambda_t^{\theta}\Lambda_s^A\xi\right).$$
    It is naturally associated with the derivation 
    $$\DC=\left(\begin{array}{cc}
   	0 & 0 \\ \xi & A
   	\end{array}\right), \hspace{.5cm}\mbox{ satisfying }\hspace{.5cm} (d\varphi_s)_{(0, 0)}=\rme^{s\DC}, \hspace{.5cm}\forall s\in\R.$$
    
   	Since $\xi\in\R^2$ and $A\in\mathfrak{gl}(2, \R)$ determine a linear vector field $\XC$, when necessary we will write $\XC=(\xi, A)$ in order to emphasize such elements.

 On the other hand, by \cite[Section 3]{DSAy1}, the exponential maps $\exp:\fg(\theta)\rightarrow G(\theta)$
is given as 
$$
	\exp (\alpha, \eta)=\left\{\begin{array}{cc}
	(0, \eta)  & \mbox{ if }\alpha=0\\
	\left(\alpha, \frac{1}{\alpha}\Lambda^{\theta}_{\alpha}\eta\right) & \mbox{ if }\alpha\neq 0		
	\end{array}\right..
	$$
 Therefore, if 
 $$L_{(t, v)}:G(\theta)\rightarrow G(\theta), \hspace{1cm}L_{(t, v)}(\hat{t}, \hat{v})=(t, v)(\hat{t}, \hat{v})=(t+t', v+\rho_t\hat{v})$$ stands for the left-translation,  the fact that 
 $$(dL_{(t, v)})_{(\hat{t}, \hat{v})}(\alpha, \eta)=(\alpha, \rho_{t}\eta),$$
implies that the flow $\{\varphi^Y_s\}_{s\in\R}$ of $Y^L$ is given by 
$$\varphi^Y_s(t, v)=L_{(t, v)}(\exp s(\alpha, \eta))=(t, v)\exp s(\alpha, \eta).$$

	\begin{remark}
		\label{exp}
		From the formula of the exponential above, it is not hard to see that, if $(a, w)\neq(0, 0)$, then
		$$\exp s_1(\alpha, \eta)=\exp s_2(\alpha, \eta)\hspace{.5cm}\implies \hspace{.5cm} s_1=s_2.$$		Such simple fact will be useful ahead.
\end{remark}

		\subsection{Simple ARS's on $G(\theta)$}

  We define now simple almost-Riemannian structures on the groups $G(\theta)$ and state their main properties. We start with the definition of invariant distributions.

		\begin{definition}
		    A 2D left-invariant distribution on $G(\theta)$ is given by the map 
		$$\Delta^L:G(\theta)\rightarrow TG(\theta), \hspace{1cm} \Delta^L(t, v)=(dL_{(t, v)})_{(0, 0)}\Delta,$$
		where $\Delta\subset\mathfrak{g}(\theta)$ is a 2D vector subspace.
		\end{definition}

  In what follows, we will always assume that a 2D left-invariant distribution $\Delta^L$ is endowed with a left-invariant Riemannian metric. That is done by considering, on $\Delta$, an inner product $\langle\cdot, \cdot\rangle_{\Delta}$ and defining 
		$$\forall X, Y\in \Delta^L(t, v), \;\;\; \langle X, Y\rangle_{(t, v)}:=\langle (dL_{(t, v)^{-1}})_{(t, v)}X, (dL_{(t, v)^{-1}})_{(t, v)}Y\rangle_{\Delta}.$$

  Let $\XC$ be a linear vector field and consider the family of vectors $\Sigma=\{\XC, \Delta^L\}$. We say that $\Sigma$ satisfies the {\bf Lie algebra rank condition (LARC)} if (at least) one of the following conditions is satisfied:

  	\begin{itemize}
				\item[(i)] $\Delta$ is not a subalgebra of $\fg(\theta)$;
				\item[(ii)] $\Delta$ is a subalgebra of $\fg(\theta)$ and $\DC\Delta\not\subset\Delta$, where $\DC$ is the derivation associated with $\XC$.
			\end{itemize}

  Next we define the concept of almost-Riemannian structures on $G(\theta)$.
		
		\begin{definition}
			\label{ARS}
			A simple {\bf almost-Riemannian structure (shortly ARS)} $\Sigma$ on $G(\theta)$ is defined by the family $\Sigma=\{\XC, \Delta^{L}\}$, where $\XC$ is a linear vector field and $\Delta^{L}$ is a 2D left-invariant distribution $\Delta$ endowed with a left-invariant Euclidean metric satisfying  
			\begin{itemize}
				\item[(i)] The set $\{(t, v)\in G(\theta);\;\;\XC(t, v)\notin\Delta^L(t, v)\}$ is nonempty;
				\item[(ii)] $\Sigma$ satisfies the LARC.
			\end{itemize}
   The metric on $G(\theta)$ associated with $\Sigma$ is defined by declaring that $\XC$ is unitary and orthogonal to $\Delta^L$ at every point.
			
		\end{definition}

		The {\bf singular points} $\ZC$ is the set of points $(t, v)\in G(\theta)$ where the previous metric fails to be Riemannian, that is
		$$\ZC:=\{(t, v)\in G(\theta); \;\XC(t, v)\in \Delta^L(t, v)\}.$$
		Since any linear vector In particular, the set of singularities $\ZC_{\XC}$ of $\XC$ is contained in $\ZC$.

		Any ARS $\Sigma$ define an {\bf almost-Riemannian norm} on $G(\theta)$ given by 
		$$\|X\|_{\Sigma, (t, v)}=\min\left\{\sqrt{\alpha_0^2+\alpha_1^2+\alpha_2^2}; \;\alpha_0\XC(t, v)+\alpha_1Y_1^L(t, v)+\alpha_2Y^L_2(t, v)=X\right\},$$
		where $\{Y_1, Y_2\}$ is an orthonormal basis of $\Delta$. It holds that $\|X\|_{\Sigma, (t, v)}=\infty$ when $(t, v)\in\ZC$ and $X\notin\Delta(t, v)$. 

    One of the interest in the study of ARSs is to analyze the possible isometries of the almost-Riemannian norm. Precisely, we say that a diffeomorphism $\psi:G(\theta)\rightarrow G(\theta)$ is an {\bf isometry} between ARSs $\Sigma_1$ and $\Sigma_2$ if   
		$$\forall (t, v)\in G(\theta), Z\in T_{(t, v)}G(\theta)\hspace{1.5cm}\|(d\psi)_{(t, v)}Z\|_{\Sigma_2, \phi(t, v)}=\|Z\|_{\Sigma_1, (t, v)}.$$

Let us denote by $\mathrm{Iso}_{\theta}(\Sigma_1; \Sigma_2)$ the group of isometries between the ARSs $\Sigma_1$ and $\Sigma_2$. Following \cite[Theorem 1]{PJAyGZ}, any isometry $\psi\in \mathrm{Iso}_{\theta}(\Sigma_1; \Sigma_2)$ can be decomposed as $\psi=L_{(t, v)}\circ\psi_0$, where $(t, v)\in \ZC_2$ and $\psi_0\in \mathrm{Iso}_{\theta}(\Sigma_1; \Sigma_2)_0$, with 
		$$\mathrm{Iso}_{\theta}(\Sigma_1; \Sigma_2)_0=\left\{\psi_0\in \mathrm{Iso}_{\theta}(\Sigma_1; \Sigma_2); \;\psi_0(0, 0)=(0, 0)\right\}.$$
		Therefore, in order to understand the group of isometries $\mathrm{Iso}_{\theta}(\Sigma_1; \Sigma_2)$ it is enough to study the singular locus $\ZC_2$ of $\Sigma_2$ and the subgroup $\mathrm{Iso}_{\theta}(\Sigma_1; \Sigma_2)_0$.
		
		Moreover, for any isometry $\psi\in \mathrm{Iso}_{\theta}(\Sigma_1; \Sigma_2)_0$ it holds that
		\begin{equation}
		\label{conjugation}
		\psi_*\Delta^L_1=\Delta_2(\psi)\hspace{.5cm}\mbox{ and }\hspace{.5cm}\psi\circ\varphi^1_s=\varphi^2_{\pm s}\circ\psi,
		\end{equation}
		where $\{\varphi_s^i\}_{s\in\R}$ is the flow associated with $\XC_i$ (see \cite[Theorem 2]{PJAyGZ}). Consequently,   
		$$\psi(\ZC_1)=\ZC_2\hspace{.5cm}\mbox{ and }\hspace{.5cm}\psi(\ZC_{\XC_1})=\ZC_{\XC_2}.$$
		
		Since outside the singular locus the metric is Riemannian, the restriction 
		$$\psi:G(\theta)\setminus \ZC_1\rightarrow G(\theta)\setminus \ZC_2,$$
		is an isometry between Riemannian manifolds. By Myers-Steenrod Theorem \cite[Theorem 8]{Myers} such restriction is of class $\CC^{\infty}$ on the connected components of $G(\theta)\setminus\ZC_1$. This fact will allows us to study the differential of an isometry between ARSs. 

  In what follows we present the main properties concerning the singular locus $\ZC$ of an ARS. For the proof the reader can consult \cite[Section 4]{DSAyD2}. 
  
 \begin{theorem}
 \label{properties}
     Let $\Sigma=\{\XC, \Delta\}$ be an ARS on $G(\theta)$ with $A\neq 0$ and $\ZC$ it singular locus. It holds:
\begin{enumerate}
    \item $\ZC$ is a closed, connected, embedded submanifold of $G(\theta)$;
    \item $G(\theta)\setminus \ZC$ has exactly two connected components;
    
    \item If $(t, v)\in G(\theta)\setminus \ZC$, the curves
    $$s\in\R\mapsto (t, v)\exp s(\alpha, \eta)\in G(\theta)\hspace{.5cm}\mbox{ and }\hspace{.5cm}s\in\R\mapsto \varphi_s(t, v)\in G(\theta),$$
    remain in the same connected component of $G(\theta)\setminus \ZC$ of they cross $\ZC$ discretely;
\end{enumerate}
     
 \end{theorem}

 \begin{remark}
     It is important to remark that by Theorem \ref{prop1}, the image of the curves 
     $$s\in\R\mapsto (t, v)\exp s(\alpha, \eta)\in G(\theta)\hspace{.5cm}\mbox{ and }\hspace{.5cm}s\in\R\mapsto \varphi_s(t, v)\in G(\theta),$$
     starting at points outside the singular locus, remains outside, up to a discrete set. 
 \end{remark}

	\section{Isometries of ARSs}

In this section we study the isometries between a class of ARSs on the groups $G(\theta)$. We show that for such class, the isometries are actually group automorphisms. The main difficult in proving this fact is that we cannot assume smoothness over the singular locus. In order to avoid this problem, we use smoothness outside the singular locus (assured by \cite[Theorem 8]{Myers}), jointly with curves crossing the singular locus.

        We start with a result, whose proof can be found at \cite[Proposition 3.3]{DSAyD2}, showing automorphisms of the group can be seen as isometries between ARSs. 
		
		\begin{proposition}
			\label{difeo}
			Let $\Sigma=\{\XC=(\xi, A), \Delta^L\}$ be an ARS on $G(\theta)$ and 
			$$\psi(t, v)=(\varepsilon t, Pv+\varepsilon\Lambda_{\varepsilon t}^{\theta}\eta)\;\;\mbox{ an automophism of }G(\theta).$$
			The family 
			$$\Sigma_{\psi}=\left\{\XC_{\psi}=\left( P^{-1}(\varepsilon\xi+ A\eta), P^{-1}AP\right), \Delta^L_{\psi}:=(\psi_*)^{-1}\Delta^L\right\},$$
			is an ARS and $\psi$ is an isometry between $\Sigma_{\psi}$ and $\Sigma$, where the left-invariant metric on $\Delta_{\psi}$ is the one that makes $(d\phi)_{(0, 0)}\bigl|_{\Delta_{\phi}}$ an isometry.
		\end{proposition}

In \cite[Theorem 4]{PJAyGZ} the authors proved that on nilpotent Lie groups, the only isometries between ARSs, fixing the identity element, are automorphisms of the group. Such fact is not necessarily true for the groups $G(\theta)$ as Example \ref{example} ahead shows. Despite this fact, in this section we present a class of ARSs in which the only possible isometries are in fact automorphisms.

\begin{example}
\label{example}
   	Consider now $G(\theta)$ and the simple ARS $\Sigma=\{\XC, \Delta^L\}$, where
   	$$\theta=\begin{pmatrix}
   		1 & 0   \\
   		0 & 0   
   	\end{pmatrix},\;\;\;\;\;\XC=\left((0, 0), \left(\begin{array}{cc}
   		0 & 1 \\  0 & 0
   	\end{array}\right)\right)\;\;\;\;\mbox{ and }\;\;\;\;\alpha=\{(1,0),(0, \mathbf{e}_2)\},$$
  	is an orthonormal basis of $\Delta$. It is not hard to see that $\Sigma$ satisfies the LARC. Moreover, for any $(t, v)\in G(\theta)$, it holds
  	$$(dL_{(t, v)})_{(0, 0)}(1, 0)=(1, 0)\;\;\;\mbox{ and }\;\;\;(dL_{(t, v)})_{(0, 0)}(0, \mathbf{e}_2)=(0, \mathbf{e}_2).$$
  	
  	It is not hard to see that the map,
  	$$\psi: G(\theta)\rightarrow G(\theta), \;\;\;(t, v)\mapsto (-t, v),$$
  	is a diffeomorphism of $G(\theta)$ and, by \ref{auto}, it holds that $\psi\notin\mathrm{Aut}(G(\theta))$. On the other hand,
  	$$\psi\left(\XC(t, v)\right)=\psi\left(0, Av\right)=(0, Av)=\XC(-t, v)=\XC(\psi(t, v)),$$
  	$$\psi(1, 0)=(-1, 0)\;\;\;\mbox{ and }\;\;\;\psi(0, \mathbf{e}_2)=(0, \mathbf{e}_2),$$
  	and since $\psi_*=\psi$, we conclude that $\psi$ carries the orthonormal frame $\{\XC, (1, 0), (0, \mathbf{e}_2)\}$ onto the orthonormal frame $\{\XC, -(1, 0), (0, \mathbf{e}_2)\}$ showing that $\psi$ is an isometry.
   \end{example}

	\bigskip
	
	 Let  $\psi:G(\theta)\rightarrow G(\theta)$ be an isometry between ARSs $\Sigma_1$ and $\Sigma_2$ and consider $f:G(\theta) \rightarrow \R$ and $g:G(\theta)\rightarrow\R^2$ the coordinate functions of $\psi$. Write, on the canonical basis,
	$$\psi_*=\left(\begin{array}{cc}
	\partial_1 f & (\partial_2 f)^T \\ \partial_1 g & \partial_2g 
	\end{array}\right),$$
	where for $(t, v)\in G(\theta)$, $\partial_2f(t, v)$ is the gradient vector of the map partial map $v\in \R^2\mapsto f(t, v)\in \R$.

	If $\{\varphi_s^i\}_{t\in\R}$ is the flow associated with the linear vector fields of $\Sigma_i$, we have by (\ref{conjugation}) that   
	$$\psi\circ \varphi_s^1=\varphi^2_s\circ\psi,$$
	where we consider only the positive case, since the negative is analogous. Using the expression for the flows of $\XC_i$, allow us to obtain
	\begin{equation}
	\label{com}
	\begin{array}{l}
	   \hspace{-1cm}f\left(t, \rme^{sA_1}v+\Lambda_s^{A_1}\Lambda_t^{\theta}\xi_1\right)=f(t, v)\hspace{.5cm} \mbox{ and } \\
	     \hspace{3cm} g\left(t, \rme^{sA_1}v+\Lambda_s^{A_1}\Lambda_t^{\theta}\xi_1\right)=\rme^{sA_2}g(t, v)+\Lambda_s^{A_2}\Lambda_{f(t, v)}^{\theta}\xi_2.
	\end{array}
	\end{equation}

 \subsection{Rank two ARSs}

 	In this section we introduce the class of rank two ARSs and prove that the only isometries between elements of such class are automorphisms of the group. This result allow us to classify, up to automorphisms, the only possible rank two ARS on the Lie groups $G(\theta)$ under consideration.

\begin{definition}
	\label{ranktwo}
		Let $\Sigma=\{\XC, \Delta^L\}$ to be a simple ARS on $G(\theta)$. We say that $\Sigma$ is a {\bf rank two} ARS if the derivation $\DC$ associated with $\XC$ has rank two. Equivalently, if $\XC=(\xi, A)$ then $\Sigma$ has rank two if and only if
			\begin{equation}	
		\R^2=\mathrm{Im}A+\R\xi.
		\end{equation}		
		\end{definition}

  \bigskip

 By the conjugation formula (\ref{conjugation}) one gets direct that rank two ARSs are preserved by isometries. Moreover, any derivation in (\ref{der}) has rank maximum equal two, and hence, the set of rank two derivations of $\fg(\theta)$ is open and dense in $\mathrm{Der}(\fg(\theta))$. On the other hand, since $G(\theta)$ is simply connected, the set of linear vector fields is isomorphic to $\mathrm{Der}(\fg(\theta))$ (see \cite[Theorem 2.2]{AyTi}), which allows us to conclude that ``most'' ARSs on $G(\theta)$ are of rank two.

	Next we state our main result. 
		\begin{theorem}
		\label{main}
		The only isometries between rank two ARSs on $G(\theta)$ that fixes the identity element are the automorphisms.
	\end{theorem}	
	
    The proof of Theorem \ref{main} is divided in the next three sections. First we show that, if $\psi$ preserves one left-invariant vector field of the associated distribution, then $\psi$ is in fact an automorphism of $G(\theta)$. This result simplifies our problem to search for only one  left-invariant vector which is preserved by the isometry. In the second part, we show that the differential of any rank two isometry lets the nilradical $\fn(\theta)$ of $\fg(\theta)$ invariant, or equivalently, if $\psi=(f, g)$ then $\partial_2f\equiv 0$. In the third and last section, we show that if $\psi_*$ keeps the nilradical invariant, then it preserves unitary left-invariant vector field generating the intersection of the nilradical with the distribution of the ARS. As a consequence, $\psi$ is an automorphism by the first section.

    \subsection{The fundamental lemma}

	 We use this section to prove a technical lemma which gives us a sufficient condition for an isometry between rank two ARS to be a group automorphism.

	\begin{lemma}[Fundamental Lemma]
		\label{fundamentallemma}
		Let $\Sigma_i=\{\XC_i, \Delta^L_i\}, i=1, 2$ be simple ARSs on the $G(\theta)$ with rank two. If  $\psi\in\mathrm{Iso}_{\theta}\left(\Sigma_1;\Sigma_2\right)_0$ satisfies 
		\begin{equation}
		\label{fundamentalequation}
		\forall (t, v)\in G(\theta);\;\;(d\psi)_{(t, v)}X^L(t, v)=(dL_{\psi(t, v)})_{(0, 0)}(d\psi)_{(0, 0)}X,
		\end{equation}
	     for some nonzero vector $X\in \Delta_1$, then $\psi\in\mathrm{Aut}(G(\theta))$. 
	\end{lemma}
	
	\begin{proof} Let $(\sigma, \eta)\in\Delta_1$ with $\sigma\neq 0$. The automorphism, 
 $\hat{\psi}(t, v)=\left(t, v-\sigma^{-1}\Lambda_t^{\theta}\eta\right),$ is such that 
 $$(d\hat{\psi})_{(0, 0)}(\sigma, \eta)=(\sigma, 0),$$
and by Proposition (\ref{difeo}) we can assume w.l.o.g. that $(1, 0)\in \Delta_1$. Under this assumption, we prove the lemma in four steps:
		\bigskip

		{\bf Step 1:} If (\ref{fundamentalequation}) holds for a nonzero $X\in\Delta_1$, then it holds fol all $Z\in\Delta_1$, that is, 		
		 $$ (d\psi)_{(t, v)}Z^L(t, v)=(dL_{\psi(t, v)})_{(0, 0)}(d\psi)_{(0, 0)}Z, \hspace{1cm}\forall (t, v)\in G(\theta), Z\in\Delta_1.$$
		
		Let $Y\in\Delta_1$ satisfying 
		$$\langle X, Y\rangle_{\Delta_1}=0, \hspace{.5cm}\mbox{  and }\hspace{.5cm}\|Y\|_{\Sigma_1, (0, 0)}=\|X\|_{\Sigma_1, (0, 0)}.$$
		By linearity, it is enough to show that relation (\ref{fundamentalequation}) holds for $Y$. From the left-invariance of the metric in $\Delta^L_1$, we get that $\{X^L(t, v), Y^L(t, v)\}$ is an orthogonal basis of $\Delta_1^L(t, v)$, satisfying 
		$$\|Y^L(t, v)\|_{\Sigma_1, (t, v)}=\|Y\|_{\Sigma_1, (0, 0)}=\|X\|_{\Sigma_1, (0, 0)}=\|X^L(t, v)\|_{\Sigma_1, (t, v)}.$$
		Since $\psi$ is an isometry, the vector $(d\psi)_{(t, v)}Y^L(t, v)$ belongs to $\Delta_2^L(\psi(t, v))$, is orthogonal to $(d\psi)_{(t, v)}X^L(t, v)$ and
		$$\|(d\psi)_{(t, v)}Y^L(t, v)\|_{\Sigma_2, \psi(t, v)}=\|Y^L(t, v)\|_{\Sigma_1, (t, v)}=\|Y\|_{\Sigma_1, (0, 0)}.$$
		On the other hand, the left-invariance of the metric in $\Delta^L_2$ implies that the vector $(dL_{\psi(t, v)})_{(0, 0)}(d\psi)_{(0, 0)}Y$ also belongs to $\Delta_2(\psi(t, v))$, is orthogonal to $(dL_{\psi(t, v)})_{(0, 0)}(d\psi)_{(0, 0)}X$ and 
		$$\|(dL_{\psi(t, v)})_{(0, 0)}(d\psi)_{(0, 0)}Y\|_{\Sigma_2, \psi(t, v)}=\|(d\psi)_{(0, 0)}Y\|_{\Sigma_2, (0, 0)}=\|Y\|_{\Sigma_1, (0, 0)}.$$ 
		Since by hypothesis, 
		$$(d\psi)_{(t, v)}X^L(t, v)=(dL_{\psi(t, v)})_{(0, 0)}(d\psi)_{(0, 0)}X,$$
		the fact that $\dim(\Delta_2^L(\psi(t, v))=2$ forces that 
		\begin{equation}
		\label{blabla}
		(d\psi)_{(t, v)}Y^L(t, v)=\varepsilon(t, v)(dL_{\psi(t, v)})_{(0, 0)}(d\psi)_{(0, 0)}Y, \;\;\;\mbox{ where }\;\;\;\varepsilon(t, v)=\pm 1.
		\end{equation}
		Moreover, by orthogonality, we obtain that 
		$$\varepsilon(t, v)=\frac{\left\langle(d\psi)_{(t, v)}Y^L(t, v), (dL_{\psi(t, v)})_{(0, 0)}(d\psi)_{(0, 0)}Y\right\rangle_{\Sigma_2, \psi(t, v)}}{\|Y\|_{\Sigma_1, (0, 0)}},\hspace{1cm}\forall (t, v)\in G(\theta)\setminus\ZC_1,$$
		showing that $\varepsilon$ is a continuous function on $G(\theta)\setminus\ZC_1$ and hence, constant on each connected component of $G(\theta)\setminus\ZC_1$. Now, the rank two assumption implies in particular that $A_1\not\equiv 0$ and hence, by Theorem \ref{properties},  $G(\theta)\setminus\ZC_1$ has two connected $\CC_1^+$ and $\CC_1^-$. Let us define $\varepsilon^+:=\varepsilon|_{\CC_1^+}$ and $\varepsilon^-:=\varepsilon|_{\CC_1^-}$.

        Let $(t, v)\in\CC_1^+$ and assume that the exponential curve $s\in\R\mapsto (t, v)\exp sY$ intersects the singular locus $\ZC_1$.
        
        Define the sets 
	    $$I^+:=\{s\in\R; \; (t, v)\exp sY\in\CC_1^+\}\hspace{.5cm}\mbox{ and }\hspace{.5cm}I^-:=\{s\in\R; \; (t, v)\exp sY\in\CC_1^-\}.$$
        It holds that $I^{+}$ and $I^-$ are open sets, $I^+\neq\emptyset$ and, by Theorem \ref{properties}, $\R\setminus (I^+\cup I^-)$ is discrete.
	     The curves
	    $$\gamma_{\pm}:I^{\pm}\rightarrow G, \hspace{1cm} \gamma_{\pm}(s):=\psi\left((t, v)\exp sY\right).$$
	    are differentiable and by equation (\ref{blabla}) satisfies 
	    $$\frac{d}{ds}\gamma_{\pm}(s)=(d\psi)_{(t, v)\exp s Y}Y^L((t, v))\exp sY)=(\varepsilon^{\pm}Z)^L(\psi((t, v)\exp sY))=(\varepsilon^{\pm}Z)^L(\gamma_{\pm}(s)),$$
	    where for simplicity $Z:=(d\psi)_{(0, 0)}Y$. Therefore, $\gamma_{\pm}(s)$ coincides with the solution of the ODE defined by the left-invariant vector field $\varepsilon^{\pm} Z^L$ on the open set  $I^{\pm}\times\CC^{\pm}_1$. By uniqueness we get that 
	    $$\psi((t, v)\exp sY)=\psi(t, v)\exp s\varepsilon^{\pm}Z, \hspace{1cm}\forall s\in I^{\pm}.$$
	    Let us assume that $I^-\neq\emptyset$. In this case, there exists $s_0\in \R\setminus I^+\cup I^-$ and $\delta>0$ such that 
        $$(t, v)\exp sY\in\CC_1^+, s\in (s_0, s_0+\delta)\hspace{.5cm}\mbox{ and }\hspace{.5cm}(t, v)\exp sY\in\CC_1^-, s\in (s_0-\delta, s_0).$$
	    By continuity, we get that 
	    $$\psi(t, v)\exp s_0\varepsilon^{+}Z=\psi((t, v)\exp s_0Y)=\psi(t, v)\exp s_0\varepsilon^{-}Z$$
	    implying that $\exp s_0\varepsilon^{+}Z=\exp s_0\varepsilon^{-}Z$ and, by Remark \ref{exp}, $\varepsilon^+=\varepsilon^-$.
	    
	    On the other hand, if $I^-=\emptyset$, let us consider an open set $(t, v)\in U\subset \CC_1^+$ and $s_0\in\R\setminus I^+$. The open set $U\exp s_0Y$ intersects $\ZC_1$ at the point $(t, v)\exp s_0Y$ and, since $\ZC_1$ is a 2D embedded manifold, we have that 
	    $$U\exp s_0Y\cap\CC_1^-\neq\emptyset.$$
	    In particular, there exists $(\hat{t}, \hat{v})\in\CC_1^+$ such that $(\hat{t}, \hat{v})\exp s_0Y\in\CC_1^-$. As a consequence, both of the sets 
	    $$\hat{I}^+:=\{s\in\R; \; (\hat{t}, \hat{v})\exp sY\in\CC_1^+\}\hspace{.5cm}\mbox{ and }\hspace{.5cm}\hat{I}^-:=\{s\in\R; \; (\hat{t}, \hat{v})\exp sY\in\CC_1^-\},$$
	    are nonempty. By the previous analysis, with $(\hat{t}, \hat{v})$ instead of $(t, v)$, we conclude that $\varepsilon^+=\varepsilon^-$.

        An analogous analysis, shows also that $\varepsilon^+=\varepsilon^-$ if for some $(t, v)\in\CC^-$ the exponential curve $s\in\R\mapsto (t, v)\exp sY$ intersects $\ZC_1$.

        Let us assume that for all $(t, v)\in G(\theta)\setminus\ZC_1$ the curve  $s\in\R\mapsto (t, v)\exp s(a, w)$ is contained in $G(\theta)\setminus\ZC_1$, that is, the exponential curves do not leave the connected component it started. As previously, for all $(t, v)\in\CC_1^{\pm}$ we get
	    $$\psi((t, v)\exp sY)=\psi(t, v)\exp s\varepsilon^{\pm}Z, \hspace{1cm}\forall s\in\R.$$
	    Using again that $\ZC_1$ is an embedded 2D manifold, there always exists a continuous curve $\gamma:(-\delta, \delta)\rightarrow G(\theta)$ crossing the singular locus $\ZC_1$, that is, 
	    $$\gamma(0)\in\ZC, \hspace{.5cm}\gamma(-\delta, 0)\subset\CC_1^-\hspace{.5cm}\mbox{ and }\hspace{.5cm}\gamma(0, \delta)\subset \CC_1^+.$$
     Therefore, for all $s\in\R$ it holds that 
        $$\psi(\gamma(\tau)\exp sY)=\psi(\gamma(\tau))\exp s\varepsilon^+Z,\hspace{.5cm}\tau\in (0, \delta)\hspace{.5cm}\mbox{ and }$$
        $$\hspace{.5cm}\psi(\gamma(\tau)\exp sY)=\psi(\gamma(\tau))\exp s\varepsilon^-Z, \hspace{.5cm}\tau\in (-\delta, 0).$$
        Taking the limit $\tau\rightarrow 0$ from both sides, allows us to conclude that 
        $$\psi(\gamma(0))\exp s\varepsilon^+Z=\psi(\gamma(0)\exp sY)=\psi(\gamma(0))\exp s\varepsilon^-Z, \hspace{1cm}\forall s\in\R,$$
        implying that $\varepsilon^+=\varepsilon^-$ and concluding the proof of Step 1.
        
		\bigskip

		{\bf Step 2:} It holds that $f(t, v)=at$ and $g(t, v)=g(0, v)+\Lambda_t^{a\theta}\partial_1g(0, 0)$.
		
		\bigskip
		
		Since $(dL_{(t, v)})_{(0, 0)}(1, 0)=(1, 0)\in\Delta_1$ we have that 
		$$\Bigl(\partial_1f(t, v), \partial_1g(t, v)\Bigr)=(d\psi)_{(t, v)}(1, 0)=(dL_{\psi(t, v)})(d\psi)_{(0, 0)}(1, 0)$$
		$$=(dL_{\psi(t, v)})\Bigl(\partial_1f(0, 0), \partial_1g(0, 0)\Bigr)=\Bigl(\partial_1f(0, 0), \rho_{f(t, v)}\partial_1g(0, 0)\Bigr),$$
		implying that 
		\begin{equation}
		\label{um}
		\partial_1f(t, v)=\partial_1f(0, 0)\;\;\;\mbox{ and }\;\;\;\partial_1g(t, v)=\rho_{f(t, v)}\partial_1g(0, 0).
		\end{equation}
		Analogously, for $(0, \eta)\in\Delta\cap\fn(\theta)$, with $\eta\neq 0$,
		$$\left(\langle\partial_2f(t, v), \rho_t\eta\rangle_{\R^2}, \partial_2g(t, v)\rho_t\eta\right)=(d\psi)_{(t, v)}(0, \rho_t\eta)=(dL_{\psi(t, v)})_{(0, 0)}(d\psi)_{(0, 0)}(0, \eta)$$
		$$=(dL_{\psi(t, v)})_{(0, 0)}\Bigl(\langle \partial_2f(0, 0), \eta\rangle_{\R^2}, \partial_2g(0, 0)\eta\Bigr)=\Bigl(\langle\partial_2f(0, 0), \eta\rangle_{\R^2}, \rho_{f(t, v)}\partial_2g(0, 0)\eta\Bigr),$$
		implying that 
		\begin{equation}
		\label{dois}
		\langle\partial_2f(t, v), \rho_t \eta\rangle_{\R^2}=\langle\partial_2f(0, 0), \eta\rangle_{\R^2}\hspace{.5cm}\mbox{ and }\hspace{.5cm}\partial_2g(t, v)\rho_t \eta=\rho_{f(t, v)}\partial_2g(0, 0)\eta.
		\end{equation}
		
		Integration of the first equation in (\ref{um}) on the variable $t\in\R$ gives us that 
		$$f(t, v)=at+h(v), \hspace{.5cm}\mbox{ where }\hspace{.5cm} a=\partial_1f(0, 0)\hspace{.5cm}\mbox{ and }\hspace{.5cm}h(v)=f(0, v).$$
		Using the previous on the first equation of (\ref{dois}) allows us to obtain 
  \begin{equation}
\label{equi}		
  \langle\partial_2f(t, v), \rho_t \eta\rangle_{\R^2}=\langle\nabla h(v), \rho_t \eta\rangle_{\R^2}\stackrel{(\ref{dois})}{=} \underbrace{\langle\nabla h(0), \eta \rangle_{\R^2}}_{=\langle\partial_2f(0, 0), \eta\rangle_{\R^2}}\stackrel{(\ref{dois})}{=}\langle\nabla h(v), \eta \rangle_{\R^2}=\langle\partial_2f(0, v), \eta\rangle_{\R^2}    
  \end{equation}
		implying that  $\langle\nabla h(v), \eta-\rho_t\eta\rangle_{\R^2}=0$ for all $t\in\R$, $v\in\R^2$.
  
		On the other hand, if $\XC_1=(\xi_1, A_1)$ , then
		\begin{equation}
		\label{h}
		 f\circ\varphi_s^1=f\hspace{.5cm}\stackrel{(\ref{com})}{\implies}
\hspace{.5cm}h(\rme^{sA_1}v+\Lambda_t^{\theta}\Lambda_s^{A_1}\xi_1)=h(v),\hspace{.5cm}\forall s, t\in\R
		\end{equation}
	gives us that,
		$$h(\rme^{sA_1}v)=h(v)\hspace{.5cm}\implies\hspace{.5cm} \nabla h(\rme^{sA_1}v)=\rme^{-sA_1^{\intercal}}\nabla h(v),$$
	and hence,
		$$\langle\nabla h(v), \eta \rangle_{\R^2}\stackrel{(\ref{equi})}{=}\langle\nabla h(0), \eta \rangle_{\R^2}\stackrel{(\ref{equi})}{=}\langle\nabla h(\rme^{sA_1}v), \eta\rangle_{\R^2}=\langle \rme^{-sA_1^{\intercal}}\nabla h(v), \eta\rangle_{\R^2}=\langle \nabla h(v), \rme^{-sA_1}\eta\rangle_{\R^2},$$
  which gives us  $\langle\nabla h(v), \eta-\rme^{-sA_1}\eta) \rangle_{\R^2}=0,$ for all $s\in\R$, $v\in\R^2$. Lastly, derivation of (\ref{h}) first at $t=0$ and then at $s=0$, gives us $\langle\nabla h(v), \xi_1\rangle_{\R^2}=0,$ for all $v\in\R^2$. 
  
  Now, if $\nabla h(v)\neq 0$ for some $v\in\R^2$, we must have that the set 
$$\{\eta-\rme^{-sA_1}\eta, \eta-\rho_t\eta, \xi_1\}\hspace{.5cm}\mbox{ is linearly dependent }\hspace{.5cm}\forall s, t\in\R.$$

Derivation of $s=0$ gives us that $\{A_1\eta, \xi_1\}$ is also linearly dependent. By the rank two hypothesis, the only way this can happen is if $\det A_1\neq 0$. On the other hand, if $\det A_1\neq 0$, a second derivative in $s=0$ forces  $\eta$ to be an eigenvector of $A_1$. This, together with the commutativity of $A_1$ and $\theta$ implies that $\eta$ is also an eigenvector of $\theta$.

However, since $\Delta_1=\mathrm{span}\{(1, 0), (0, \eta)\}$ we get that $\Delta_1$ is a subalgebra and 
$$\DC(1, 0)=(0, \xi_1)\in\Delta_1\hspace{.5cm}\mbox{ and }\hspace{.5cm}\DC_1(0, \eta)=(0, A_1\eta)\in\Delta_1\hspace{.5cm}\implies\hspace{.5cm}\DC_1\Delta_1\subset\Delta_1,$$
which contradicts the fact that $\Sigma_1$.
 Therefore, 
		$$\nabla h(v)=0\hspace{.5cm}\forall v\in\R^2\hspace{.5cm}\implies\hspace{.5cm} h(v)=h(0)=f(0, 0)=0, \hspace{.5cm}\forall v\in\R^2\hspace{.5cm}\implies\hspace{.5cm}f(t, v)=at,$$
	    as stated.
		
		For the function $g$, using the expression $f(t, v)=at$ on the second equation of (\ref{um}), allows us to conclude that 
		$$\partial_1g(t, v)=\rho_{at}\partial_1g(0, 0),$$
		which by integration gives us 
		$$g(t, v)-g(0, v)=\int_0^t\partial_1g(s, v) ds=\int_0^t\rho_{as}\partial_1g(0, 0) ds=\Lambda_t^{a\theta}\partial_1g(0, 0),$$
		showing the assertion for $g$. 
		\bigskip
		
		{\bf Step 3:} For all $(t, v)\in G(\theta)$ and $s\in\R$, it holds that 
		\begin{equation}
		\label{xi}
		\partial_2g(0, v)\rho_{t}\xi_1=\rho_{at}\partial_2g(0, 0)\xi_1\hspace{.5cm}\mbox{ and }\hspace{.5cm} \partial_2g(0, v)\rho_t\rme^{-sA_1}\eta=\rho_{at}\partial_2g(0, 0)\rme^{-sA_1}\eta.
		\end{equation} 
		
		By the second equation in (\ref{com}), it holds that 
		$$g\left(t, \rme^{sA_1}v+\Lambda_s^{A_1}\Lambda_t^{\theta}\xi_1\right)=\rme^{sA_2}g(t, v)+\Lambda_s^{A_2}\Lambda_{at}^{\theta}\xi_2,$$
  where $\XC_2=(\xi_2, A_2)$. By Step 2,
		$$g\left(0, \rme^{sA_1}v+\Lambda_s^{A_1}\Lambda_t^{\theta}\xi_1\right)+\Lambda_{at}^{\theta}\nu=\rme^{sA_2}g(0, v)+\rme^{sA_2}\Lambda_{at}^{\theta}\nu+\Lambda_s^{A_2}\Lambda_{at}^{\theta}\xi_2,$$
		where for simplicity we put $\nu=\partial_1g(0, 0)$. Therefore,  
		\begin{equation}
		\label{uau}
		g\left(0, \rme^{sA_1}v+\Lambda_s^{A_1}\Lambda_t^{\theta}\xi_1\right)-\rme^{sA_2}g(0, v)=(\rme^{sA_2}-\id_{\R^2})\Lambda_{at}^{\theta}\nu+\Lambda_s^{A_2}\Lambda_{at}^{\theta}\xi_2=
		\Lambda_s^{A_2}\Lambda_{at}^{\theta}\left(\xi_2+A_2\nu\right),
		\end{equation}
  where for the last equality we used that $\rme^{sA_2}-\id_{\R^2}=\Lambda_{t}^{A_2}A_2$ and $A_2\theta=\theta A_2$. 	Derivation of equation (\ref{uau}) on $t$ gives us, by the chain rule, that 
		$$\partial_2g\left(0, \rme^{sA_1}v+\Lambda_s^{A_1}\Lambda_t^{\theta}\xi_1\right)\rho_{t}\Lambda_s^{A_1}\xi_1=a\rho_{at}\Lambda_s^{A_2}\left(\xi_2+A_2\nu\right).$$
		Since the previous equations is true for all $s\in\R$ and $(t, v)\in G$, we can substitute $v$ by $\rme^{-sA_1}\left(v-\Lambda_s^{A_1}\Lambda_t^{\theta}\xi_1\right)$ in order to obtain 
		$$\partial_2g(0, v)\rho_{t}\Lambda_s^{A_1}\xi_1=a\rho_{at}\Lambda_s^{A_2}\left(\xi_2+A_2\nu\right).$$
		Derivation of the previous equation at $s=0$ gives us 	   	   
		$$
		\partial_2g(0, v)\rho_{t}\xi_1=a\rho_{at}\left(\xi_2+A_2\nu\right)\hspace{.5cm}\implies\hspace{.5cm}\partial_2g(0, 0)\xi_1=a\left(\xi_2+A_2\nu\right),
		$$
		and proves the first equality. For the second equality, let us notice that the right-hand side of equation (\ref{uau}) does not depends on $v$. Differentiation on the variable $v$ gives us, by the chain rule, that 
  \begin{equation}
      \label{commuti}
      \partial_2g\left(0, \rme^{sA_1}v+\Lambda_s^{A_1}\Lambda_t^{\theta}\xi_1\right)\rme^{sA_1}=\rme^{sA_2}\partial_2g(0, v).
  \end{equation}
	
		On the other hand, by Step 2. 
		$$\forall (t, v)\in G(\theta), \;\;\; \partial_2g(t, v)=\partial_2g(0, v) \;\;\;\;\mbox{ and }\;\;\;\;f(t, v)=at,$$
		which by the second equation in (\ref{dois}) gives us that
		$$\partial_2g(0, v)\rho_t \eta=\rho_{at}\partial_2g(0, 0)\eta.$$
		Therefore, for any $(t, v)\in G(\theta)$ and $s\in \R$ we get
		$$\partial_2g(0, v)\rho_t\rme^{-sA_1}\eta=\partial_2g(0, v)\rme^{-sA_1}\rho_t\eta\stackrel{(\ref{commuti})}{=}\rme^{-sA_2}\partial_2g\left(0, \rme^{sA_1}v+\Lambda_s^{A_1}\Lambda_t^{\theta}\xi_1\right)\rho_t\eta$$
		$$=\rme^{-sA_2}\rho_{at}\partial_2g(0, 0)\eta=\rho_{at}\rme^{-sA_2}\partial_2g(0, 0)\eta\stackrel{(\ref{commuti})}{=}\rho_{at}\partial_2g(0, 0)\rme^{-sA_1}\eta,$$
		 showing the assertion.

		\bigskip

		{\bf Step 4:} $\psi\in \mathrm{Aut}(G(\theta))$
		
		\bigskip

We have to analyze the following possibilities:
			
			\begin{itemize}
				\item[1.] $\Delta_1$ is not a subalgebra:
				
				In this case, $\eta\neq 0$ considered as previous cannot be  an eigenvector of $\theta$, and so, there exists $t_0\in\R$ is such that $\{\eta, \rho_{t_0}\eta\}$ is a basis of $\R^2$.
From the second equation in (\ref{xi}) we get:
\begin{itemize}
    \item[(i)] For $s=0$,
    $$\forall (t, v)\in G(\theta), \hspace{1cm}\partial_2g(0, v)\rho_t\eta=\rho_{at}\partial_2g(0, 0)\eta;$$

    \item[(ii)] For $s=0$ and all $(t, v)\in G(\theta)$,
    $$\partial_2g(0, v)\rho_t\rho_{t_0}\eta=\partial_2g(0, v)\rho_{t+t_0}\eta=\rho_{a(t+t_0)}\partial_2g(0, 0)\eta=\rho_{at}\rho_{at_0}\partial_2g(0, 0)\eta=\rho_{at}\partial_2g(0, 0)\rho_{t_0}\eta,$$
    where for the last equality we used (ii) for $t=t_0$ and $v=0$.
\end{itemize}

Therefore, the maps $\partial_2g(0, v)\rho_t$ and $\rho_{at}\partial_2g(0, 0)$ coincide on a basis of $\R^2$ and hence
\begin{equation}
    \label{buu}
    \forall (t, v)\in G(\theta), \hspace{1cm}\partial_2g(0, v)\rho_t=\rho_{at}\partial_2g(0, 0).
\end{equation}
As a consequence, for $t=0$ we obtain that 
$$\forall v\in\R^2, \hspace{1cm}\partial_2g(0, v)=\partial_2g(0, 0).$$

Moreover, if $\tr\theta\neq 0$, then 
                $$\det\left(\partial_2g(0, 0)\right)\rme^{t\tr\theta}=\det\left(\partial_2g(0, 0)\circ \rho_t\right)=\det\left(\rho_{at}\circ \partial_2g(0, 0)\right)=\rme^{at\tr\theta}\det\left(\partial_2g(0, 0)\right),$$
                implying that $a=1.$ On the other hand, if $\tr\theta=0$ then necessarily $\det\theta\neq0$. Derivation of equation (\ref{buu}) gives us that  
                $$\partial_2g(0, 0)\circ\theta=a\theta\circ\partial_2g(0, 0),$$
                and hence, 
                $$\det\left(\partial_2g(0, 0)\right)\det\theta=\det\left(\partial_2g(0, 0)\circ\theta\right)=\det\left(a\theta\circ\partial_2g(0, 0)\right)=a^2\det\theta\det\left(\partial_2g(0, 0)\right),$$
                implying that $a\in\{-1, 1\}$.

				\item[2.] $\Delta_1$ is a subalgebra and $\theta\neq\id_{\R^2}$:
				
				In this case, the fact that $A_1\theta=\theta A_1$ implies that $\eta$ is also an eigenvector of $A_1$. Since $\Delta_1=\mathrm{span}\{(1, 0), (0, \eta)\}$, we conclude as in Step 2 that the LARC holds for $\Sigma_1$ if and only if the set $\{\xi_1, \eta\}$ is linearly independent.
				
				Considering $t=s=0$ in the equations in (\ref{xi}) gives us that  	
				$$\partial_2g(0, v)\eta=\partial_2g(0, 0)\eta\hspace{.5cm}\mbox{ and }\hspace{.5cm} \partial_2g(0, v)\xi_1=\partial_2g(0, v)\xi_1,$$
				which by linearity implies that, 
				$$\forall v\in\R^2, \hspace{.5cm}\partial_2g(0, v)=\partial_2g(0, 0)\hspace{.5cm}\mbox{ and }\hspace{.5cm}\partial_2g(0, 0)\circ \rho_t=\rho_{at}\circ \partial_2g(0, 0),$$
    and, as in the previous case, we conclude that $a=1$ if $\tr\theta\neq 0$ and $a\in\{-1, 1\}$ if $\tr\theta=0$.
				
				\item[3.] $\Delta_1$ is a subalgebra and $\theta=\id_{\R^2}$:
				
				In this case, we get from the second equation in (\ref{xi}) that
				$$\rme^{t}\partial_2g(0, v)\eta=\partial_2g(0, v)\rho_t\eta=\rho_{at}\partial_2g(0, 0)\eta=\rme^{at}\partial_2g(0, v)\eta,$$
				implying that $a=1$. If $\{\xi_1, \eta\}$ is linearly independent, we conclude as previously, that 
				$$\forall v\in\R^2,\hspace{1cm}\partial_2g(0, v)=\partial_2g(0, 0).$$
				On the other hand, if $\{\xi_1, \eta\}$ is linearly dependent, the LARC implies necessarily that $\eta$ cannot be an eigenvector of $A_1$. In particular, for some $s_0\in\R$, $\{\eta, \rme^{-s_0A_1}\eta\}$ is a basis of $\R^2$ and as previously we conclude that the linear maps  $\partial_2g(0, v)$ and $\partial_2g(0, 0)$ coincide on such basis, implying that $\partial_2g(0, v)=\partial_2g(0, 0)$ for all $v\in \R^2$.
			\end{itemize}
   
			In any case, we get that 
			$$f(t, v)=\varepsilon t\;\;\;\mbox{ and }\;\;\; g(t, v)=Pv+\varepsilon\Lambda_{t}^{\varepsilon\theta}\nu,$$
			where $\nu=\partial_1g(0, 0)$, $P=\partial_2g(0, 0)$ and $P\theta=\varepsilon\theta P$ with $\varepsilon=1$ if $\tr\theta\neq 0$ or $\varepsilon\in\{-1, 1\}$ if $\tr\theta=0$. Moreover, by definition,  
			$$\Lambda_t^{\varepsilon t}\nu=\int_0^t\rme^{\varepsilon s\theta}\nu ds=\varepsilon\int_0^{\varepsilon t}\rme^{\mu \theta}\nu d\mu=\varepsilon\Lambda^{\theta}_{\varepsilon t}\nu,$$
			implying that $\psi\in\mathrm{Aut}(G(\theta))$ and concluding the proof.
	\end{proof}

	\subsection{Invariance of the nilradical}
	
	In this section we show that for an isometry $\psi\in\mathrm{Iso}_{\theta}(\Sigma_1; \Sigma_2)_0$ between rank two ARSs with $\psi=(f, g)$, the coordinate function $f$ only depends on the first variable of $G(\theta)$ or equivalently, it satisfies $\partial_2f\equiv 0$.

	Let us consider an orthonormal basis $\{X_i, Y_i\}\subset\Delta_i$. Since $\psi_*\Delta_1^L=\Delta_2^L\circ\psi$, we can write uniquely
	$$\psi_*X_1^L=\alpha_1X_2^L(\psi)+\alpha_2Y_2^L(\psi),$$
	where 
	$$\alpha_1^2+\alpha_2^2=\|\psi_*X_1^L\|^2_{\Sigma_2}=\|X_1^L\|^2_{\Sigma_1}=\|X_1\|^2_{\Sigma_1}=1,$$
	and each $\alpha_i$ can be recovered from the orthonormality of the basis as
	$$\alpha_1=\langle \psi_*X_1^L, X_2^L(\psi)\rangle_{\Sigma_2}\hspace{.5cm}\mbox{ and }\hspace{.5cm}\alpha_2=\langle \psi_*X_1^L, Y_2^L(\psi)\rangle_{\Sigma_2},$$
	showing that $\alpha_i:G(\theta)\setminus\ZC_1\rightarrow\R$ are $\CC^{\infty}$ functions. The fact that $\{X_1^L, Y_1^L\}$ is an orthonormal basis of $\Delta^L_1$ and $\psi$ is an isometry allow us to write 
	$$\psi_*Y_1^L=\epsilon\left(-\alpha_2X_2^L(\psi)+\alpha_1Y_2^L(\psi)\right),$$
	where $\epsilon=\pm 1$ is constant on the connected components of $G(\theta)\setminus\ZC_1$. Since all the calculations will be taking into account the connected components of $G(\theta)\setminus\ZC_1$ we will assume that $\epsilon=1$. As one can easily see, such assumption bears no loss of generality and avoids cumbersome notation.
	
	Let us assume w.l.o.g. that   
	$$X_i=(0, \eta_i)\hspace{.5cm}\mbox{ and }\hspace{.5cm} Y_i=(\sigma_i, \nu_i), \hspace{.5cm}\mbox{ with }\hspace{.5cm} \sigma_i\neq 0.$$
	By writting $\psi=(f, g)$ and $\psi_*=\left(\begin{array}{cc}
	\partial_1f & (\partial_2 f)^T\\    \partial_1g & \partial_2g
	\end{array}\right)$, the previous considerations imply that 	
	$$\psi_*(0, \rho_t\eta_1)=\alpha_1(0, \rho_f\eta_2)+\alpha_2(\sigma_2, \rho_f\nu_2)\hspace{.5cm}\mbox{ and }\hspace{.5cm}\psi_*(\sigma_1, \rho_t\nu_1)=-\alpha_2(0, \rho_f\eta_2)+\alpha_1(\sigma_2, \rho_f\nu_2),$$
	or equivalently 
	\begin{equation}
	\label{9}
	\left\{\begin{array}{l}
	\langle\partial_2f,  \rho_t\eta_1\rangle_{\R^2}=\alpha_2\sigma_2\\
	\sigma_1\partial_1f+\langle\partial_2f,\rho_t\nu_1\rangle_{\R^2}=\alpha_1\sigma_2
	\end{array}\right.
	\hspace{.5cm}\mbox{ and }\hspace{.5cm}\left\{\begin{array}{l}
	\partial_2g\rho_t\eta_1=\rho_{f}\left(\alpha_1\eta_2+\alpha_2\nu_2\right)\\
	\sigma_1\partial_1g+\partial_2g\rho_t\nu_1=\rho_{f}\left(-\alpha_2\eta_2+\alpha_1\nu_2\right)
	\end{array}\right..
	\end{equation}
		
	Now, 
	$$\psi\circ\varphi_s^1=\varphi_s^2\circ\psi\hspace{.5cm}\implies\hspace{.5cm} f\circ\varphi_s^1=f,$$
   and consequently, by differentiation,
   \begin{equation}
   \label{eqf}
   \langle\partial_2f,A_1v+\Lambda_t^{\theta}\xi_1\rangle_{\R^2} = 0, \hspace{.5cm} \partial_1f = \partial_1f(\varphi_s^1)+\langle\partial_2f(\varphi^1_s),\rho_t\Lambda_{s}^{A_1}\xi_1\rangle_{\R^2}\hspace{.5cm}\mbox{ and }\hspace{.5cm}\partial_2f(\varphi_s^1)=\rme^{-sA_1^{\intercal}}\partial_2f.
   \end{equation}

   Now, from the equations in (\ref{9}) we obtain that 
	$$\sigma_2^2=(\sigma_2\alpha_2(\varphi_s^1))^2+(\sigma_2\alpha_1(\varphi_s^1))^2=\langle\partial_2f(\varphi_s^1), \rho_t\eta_1\rangle_{\R^2}^2+\left(\sigma_1\partial_1f(\varphi_s^1)+\langle\partial_2f(\varphi_s^1), \rho_t\nu_1\rangle_{\R^2}\right)^2$$
	$$=\langle\rme^{-sA_1^{\intercal}}\partial_2f, \rho_t\eta_1\rangle_{\R^2}^2+\left(\sigma_1\left(\partial_1f-\langle\rme^{-sA_1^{\intercal}}\partial_2f, \rho_t\Lambda_{s}^{A_1}\xi_1\rangle_{\R^2}\right)+\langle\rme^{-sA_1^{\intercal}}\partial_2f, \rho_t\nu_1\rangle_{\R^2}\right)^2$$
	$$=\langle\partial_2f, \rme^{-sA_1}\rho_t\eta_1\rangle_{\R^2}^2+\left(\sigma_1\left(\partial_1f+\langle\partial_2f, \rho_t\Lambda_{-s}^{A_1}\xi_1\rangle_{\R^2}\right)+\langle\partial_2f, \rme^{-sA_1}\rho_t\nu_1\rangle_{\R^2}\right)^2,$$
	where in the second parenthesis we used that $\rme^{-sA_1}\Lambda_{s}^{A_1}=-\Lambda_{-s}^{A_1}$.
	 
	Derivation at $s=0$ gives us
	$$0=2\underbrace{\langle\partial_2f, \rho_t\eta_1\rangle_{\R^2}}_{=\alpha_2\sigma_2}\Bigl(\langle \partial_2f, -A_1\rho_t\eta_1\rangle_{\R^2}\Bigr)+2\underbrace{(\sigma_1\partial_1f+\langle\partial_2f, \rho_t\nu_1\rangle_{\R^2})}_{=\alpha_1\sigma_2}\Bigl(-\sigma_1\langle\partial_2f, \rho_t\xi_1\rangle_{\R^2}$$
	$$+\langle\partial_2f, -A_1\rho_t\nu_1\rangle_{\R^2}\Bigr)=-2\left\langle\partial_2f, \rho_t\left(\sigma_1\sigma_2\alpha_1\xi_1+\sigma_2A_1(\alpha_2\eta_1+\alpha_1\nu_1)\right)\right\rangle_{\R^2},$$
	showing that 
    $$\partial_2f\hspace{.5cm}\mbox{ is orthogonal to }\hspace{.5cm}\rho_t\left(\sigma_1\sigma_2\alpha_1\xi_1+\sigma_2A_1(\alpha_2\eta_1+\alpha_1\nu_1)\right)\;\;\mbox{ on }\;\;G(\theta)\setminus\ZC_1.$$
     Moreover, by the first equation in (\ref{eqf}), we get that, on $G(\theta)\setminus\ZC_1$,
     \begin{equation}
    \label{LD}
    \partial_2f\neq 0\implies\hspace{.5cm} \{A_1v+\Lambda_t^{\theta}\xi_1, \rho_t(\sigma_1\sigma_2\alpha_1\xi_1+\sigma_2A_1(\alpha_2\eta_1+\alpha_1\nu_1)\}\hspace{.5cm} \mbox{ is linearly dependent}.
    \end{equation}
    
    We show that $\partial_2f\equiv0$ by analyzing the possibilities for the eigenvalues of $A_1$ in the next propositions.
   	
   	
\begin{proposition}
	\label{prop1}
	If $A_1$ has only eigenvalues with nonzero real parts, $\partial_2f\equiv 0$.
\end{proposition}

\begin{proof}
	As in the proof of Lemma \ref{fundamentallemma}, we can use Proposition \ref{difeo} and assume w.l.o.g. that $(1, 0)\in\Delta_1$. Under this assumption, we have that $\nu_1=c\eta_1$ and equations (\ref{9}) become 
	$$\langle\partial_2f, \rho_t\eta_1\rangle_{\R^2}=\alpha_2\sigma_2\hspace{.5cm}\mbox{ and }\hspace{.5cm}\partial_1f= \frac{\sigma_2}{\sigma_1}(\alpha_1-c\alpha_2).$$
	Therefore,
	$$\alpha_2(\varphi_s^1)\sigma_2=\langle\partial_2f(\varphi_s^1), \rho_t\eta_1\rangle_{\R^2}\stackrel{(\ref{eqf})}{=}\langle\rme^{-sA_1^{\intercal}}\partial_2f, \rho_t\eta_1\rangle_{\R^2}$$
	is bounded for $s\in\R$. However, the assumption that $A_1$ has only eigenvalues with nonzero real parts and $\eta_1\neq 0$ implies that $\alpha_2\equiv 0$,  and hence, $\alpha_1=1$ showing that 
	$\partial_1f$ is constant on $G\setminus\ZC_1$.
	Using the continuity of $f$ and the fact that $G(\theta)\setminus\ZC_1$ is an open and dense subset of $G(\theta)$ allow us to conclude that  
	$$\forall (t, v)\in G(\theta), \;\;\;\;f(t, v)=at+h(v), \;\;\mbox{ where }\;\; a=\sigma_2/\sigma_1\;\;\mbox{ and }\;\;h(v)=f(0, v).$$
	By doing the same analysis as the one in Step 2 of Lemma \ref{fundamentallemma}, allows us to conclude, by the LARC, that $h\equiv 0$, implying that $\partial_2f\equiv 0$ as stated.	 
\end{proof}

\begin{proposition}
If $A_1$ has a pair of pure imaginary eigenvalues, $\partial_2f\equiv 0$.	
\end{proposition}

\begin{proof}
Since $\det A_1\neq 0$, the automorphism, 
 $\hat{\psi}(t, v)=\left(t, v-\Lambda_t^{\theta}(A^{-1}\xi_1)\right),$ is well defined. Moreover, by Proposition (\ref{difeo}) it is an isometry between $\Sigma_1$ and an ARS $\hat{\Sigma}_1$ whose associated linear vector field is $\hat{\XC}_1=(0, A_1)$. Therefore, we will assume w.l.o.g. that $\xi_1=0$.
 
 Under this assumption, let us define the set 
	$$\AC:=\{(t, v)\in G(\theta); \;\;\partial_2f(t, v)\neq 0\}.$$
	The first to notice is that 
	$$\partial_2f(\varphi_s^1)=\rme^{-sA^{\intercal}}\partial_2f\hspace{.5cm}\implies\hspace{.5cm}\varphi_s^1(\AC)\subset\AC.$$
	Also, the fact that $A_1$ has a pair of imaginary eigenvalues and that $A_1\theta=\theta A_1$ imply that, on the canonical basis, $A_1=\mu R$, where $R=\left(\begin{array}{cc}
	0 & -1\\ 1 & 0
	\end{array}\right)$ and $\mu\neq 0$.
	
	Since we are assuming $\xi_1=0$, the singular locus $\ZC_1$ and the set of singularities $\ZC_{\XC_1}$ are given explicitly as 
	$$\ZC_1=\{(t, s\rho_tR\eta_1), \;\;t, s\in\R\}\hspace{.5cm}\mbox{ and }\hspace{.5cm}\ZC_{\XC_1}=\R\times\{0\}.$$ 
	Consequently, for any $(t, v)\in\ZC_1\setminus\ZC_{\XC_1}$ there exists $s_0\in\R$ such that $\varphi_{s_0}(t, v)\in G(\theta)\setminus\ZC_1$. In particular we get that  
	$$\AC\setminus\ZC_{\XC_1}\neq\emptyset\hspace{.5cm}\implies \hspace{.5cm} \AC\cap \left(G(\theta)\setminus\ZC_1\right)\neq\emptyset.$$ 
 
	Let us assume the existence of $(t_0, v_0)\in \AC\cap \left(G(\theta)\setminus\ZC_1\right)$. From relation (\ref{LD}) we have that 
	$$\alpha_2(t_0, v_0)\rho_{t_0}\eta_1+\alpha_1(t_0, v_0)\rho_{t_0}\nu_1\in\R v_0,$$ 
 or equivalently, 
 \begin{equation}
 \label{relation}
  \alpha_2(t_0, v_0)\langle\rho_{t_0}\eta_1, Rv_0\rangle_{\R^2}+\alpha_1(t_0, v_0)\langle\rho_{t_0}\nu_1, Rv_0\rangle_{\R^2}=0.   
 \end{equation}

Now, if $\nu_1=c\eta_1$, equation (\ref{relation}) implies that 
$$\bigl(\alpha_2(t_0, v_0)+c\alpha_1(t_0, v_0)\bigr)\langle\rho_{t_0}\eta_1, Rv_0\rangle=0.$$
By a simple calculation, we get that 
$$\hat{\ZC}_1=\{(t, s\rho_t\eta_1), \;t, s\in\R\},$$
is the set of points where $\langle\rho_{t_0}\eta_1, Rv_0\rangle=0$, and hence
$$\alpha_1^2(t_0, v_0)=\frac{1}{1+c^2}\hspace{.5cm}\mbox{ and }\hspace{.5cm}\alpha_2^2(t_0, v_0)=\frac{c^2}{1+c^2}, \hspace{.5cm}\forall (t_0, v_0)\in\AC\cap\left(G(\theta)\setminus(\ZC_1\cup\hat{\ZC}_1)\right),$$
where we used that $\alpha_1^2+\alpha_2^2=1$. As in Proposition \ref{prop1} the assumption $\nu_1=c\eta_1$ gives us
$$\langle\partial_2f, \rho_t\eta_1\rangle_{\R^2}=\alpha_2\sigma_2\hspace{.5cm}\mbox{ and }\hspace{.5cm}\partial_1f= \frac{\sigma_2}{\sigma_1}(\alpha_1-c\alpha_2),$$
implying, in particular, that $\partial_1f$ is constant on each connected component of $G(\theta)\setminus(\ZC_1\cup\hat{\ZC}_1)$. Using the continuity of $f$ and the fact that $G(\theta)\setminus(\ZC_1\cup\hat{\ZC}_1)$ is an open and dense subset of $G(\theta)$ allow us to conclude that  $$\forall (t, v)\in G(\theta), \hspace{.5cm}f(t, v)=at+h(v), \hspace{.5cm}\mbox{ where }\hspace{.5cm}h(v)=f(0, v).$$
	By doing the same analysis as the one in Step 2 of Lemma \ref{fundamentallemma}, we are able to conclude, by the LARC, that $h\equiv 0$, which gives us $\partial_2f\equiv 0$.

On the other hand, if $\{\eta_1, \nu_1\}$ is a linearly independent set, then $\langle\rho_{t_0}\eta_1, Rv_0\rangle_{\R^2}^2+\langle\rho_{t_0}\nu_1, Rv_0\rangle_{\R^2}^2\neq 0$ on $G(\theta)$. As a consequence, relation (\ref{relation}) implies that for all $(t_0, v_0)\in \AC\cap (G(\theta)\setminus\ZC_1)$,
	\begin{equation}
		\label{alphais}
	\alpha_1(t_0, v_0)^2=\frac{\langle\rho_{t_0}\eta_1, Rv_0\rangle_{\R^2}^2}{\langle\rho_{t_0}\eta_1, Rv_0\rangle_{\R^2}^2+\langle\rho_{t_0}\nu_1, Rv_0\rangle_{\R^2}^2}\hspace{.5cm}\mbox{ and }\end{equation}
	$$\alpha_2(t_0, v_0)^2=\frac{\langle\rho_{t_0}\nu_1, Rv_0\rangle_{\R^2}^2}{\langle\rho_{t_0}\eta_1, Rv_0\rangle_{\R^2}^2+\langle\rho_{t_0}\nu_1, Rv_0\rangle_{\R^2}^2}.$$
    Next we show that $\AC\subset\ZC_{\XC_1}$, analyzing two possibilities:
	\begin{enumerate}
		\item $\eta_1$ and $\nu_1$ are not orthogonal vectors;
		
		Let $(t, v)\in\AC\setminus\ZC_{\XC_1}$. The fact that $\rme^{sA_1}v$ is a circumference around the origin gives us that 
		$$\{\rme^{sA_1}v, \;\;s\in\R\}\cap\{s\rho_t\nu_1, \;s\in\R\}\neq\emptyset\hspace{.5cm}\implies\hspace{.5cm}\exists s_0, s_1\in\R^*; \;\;\rme^{s_0A_1}v=s_1\rho_t\nu_1.$$
		In particular, if $\eta_1$ and $\nu_1$ are not orthogonal vectors, $A_1\nu_1\notin\R \eta_1$ implying that $(t, s_1\rho_t\nu_1)\in G(\theta)\setminus\ZC_1$ and hence 
		$$\AC\ni\varphi^1_{s_0}(t, v)=(t, \rme^{s_0A_1}v)=(t, s_1\rho_t\nu_1)\in G(\theta)\setminus\ZC_1.$$
		Formula (\ref{alphais}) with $(t_0, v_0)=(t, s_1\rho_t\nu_1)$ implies
		$$\langle\rho_t\nu_1, Rv_0\rangle_{\R^2}=s_1\langle\rho_t\nu_1, R\rho_t\nu_1\rangle_{\R^2}=0\hspace{.5cm}\implies\hspace{.5cm}\alpha_2(t,  s_1\rho_t\nu_1)=0\;\;\stackrel{(\ref{9})}{\implies}\;\;\;\langle\partial_2f(t, s_1\rho_t\nu_1), \rho_t\eta_1\rangle_{\R^2}=0.$$
		On the other hand, by the first equation in (\ref{eqf}) we get that 
		$$\langle\partial_2f(t, s_1\rho_t\nu_1), A_1(s_1\rho_t\nu_1)\rangle_{\R^2}=0\hspace{.5cm}\implies\hspace{.5cm}\langle\partial_2f(t, s_1\rho_t\nu_1), \rho_tA_1\nu_1\rangle_{\R^2}=0.$$
		Since  $A_1\nu_1\notin\R \eta_1$, the set $\{\rho_t \eta_1, \rho_tA_1\nu_1\}$ is a basis of $\R^2$, which by the previous equalities imply $\partial_2f(t, s_1\rho_t\nu_1)=0$, and hence
  $$\partial_2f(t, v)\stackrel{(\ref{eqf})}{=}\rme^{-s_0A_1^{\intercal}}\partial_2f(\varphi^1_{s_0}(t, v))=\rme^{-s_0A_1^{\intercal}}\underbrace{\partial_2f(t, s_1\rho_t\nu_1)}_{=0}=0,$$
  contradicting the fact that $(t, v)\in\AC\setminus\ZC_{\XC_1}$. Therefore, if  $\eta_1$ and $\nu_1$ are not orthogonal vectors we must have that $\AC\setminus\ZC_{\XC_1}=\emptyset$.
		
		\item $\eta_1$ and $\nu_1$ are orthogonal vectors;
		
		Let $(t, v)\in\AC\setminus\ZC_{\XC_1}$. As previously, 
		$$\exists s_0, s_1\in\R^*; \;\;\;\;\AC\ni\varphi_{s_0}^1(t, v)=(t, s_1\rho_t\eta_1)\in G(\theta)\setminus\ZC_1,$$
  and, by using (\ref{alphais}) with $(t_0, v_0)=(t, s_1\rho_t\eta_1)$, allows us to obtain that $\alpha_1(t, s_1\rho_t\eta_1)=0.$
		Since $\eta_1$ and $\nu_1$ being orthogonal  is equivalent to $A_1\eta_1\in\R \nu_1$, the first equation in (\ref{eqf}) gives us that
		$$\langle\partial_2f(t, s_1\rho_t\eta_1), A_1(s_1\rho_t\eta_1)\rangle_{\R^2}=0\hspace{.5cm}\implies\hspace{.5cm}\langle\partial_2f(t, s_1\rho_t\eta_1), \rho_t\nu_1\rangle_{\R^2}=0,$$
		and hence, 
		$$\sigma_1\partial_1f(t, s_1\rho_t\eta_1)=\sigma_1\partial_1f(t, s_1\rho_t\eta_1)+\underbrace{\langle\partial_2f(t, s_1\rho_t\eta_1), \rho_t\nu_1\rangle_{\R^2}}_{=0}\stackrel{(\ref{9})}{=}\alpha_1(t, s_1\rho_t\eta_1)\sigma_2=0.$$
		Since we are assuming $\xi_1=0$, the second equation in (\ref{eqf}) implies that $\partial_1f(\varphi^1_s)=\partial_1f$ for all $s\in\R$. In particular, 
		$$\partial_1f(t, v)=\partial_1f(\varphi^1_{s_0}(t, v))=\partial_1f(t, s_1\rho_t\eta_1)=0,$$ 
		and by the arbitrariness of $(t, v)\in\AC\setminus\ZC_{\XC_1}$ we conclude that 
		$$\AC\setminus\ZC_{\XC_1}\subset \{(t, v)\in G(\theta)\setminus\ZC_{\XC_1}; \;\; \partial_1f(t, v)=0\}.$$
		Now, if $(t, v)\in (\overline{\AC}\setminus\AC)\cap \left(G(\theta)\setminus\ZC_1\right)$, the fact that $f$ restrict to $G(\theta)\setminus\ZC_1$ is $\CC^{\infty}$ implies that
		$$\partial_1f(t, v)=0\hspace{.5cm}\mbox{ and }\hspace{.5cm}\partial_2f(t, v)=0,$$
		which is a contradition to the fact that $\psi$ is a diffeomorphism on $G\setminus\ZC_1$. Therefore, $\AC\cap(G(\theta)\setminus\ZC_1)$ is open and closed in $G(\theta)\setminus\ZC_1$ and so, $\AC$ contains any connected component of $G(\theta)\setminus\ZC_1$ that it intersects. 

  On the other hand, for any $(t, v)\in G(\theta)$ with $v\neq 0$, the curve $\{\varphi_s(t, v)=(t, \rme^{sA_1}), s\in\R\}$ is a circumference on the plane $\{t\}\times\R^2$ with center $(t, 0)$ and radius $|v|_{\R^2}$. Since the intersection 
  $\left(\{t\}\times\R^2\right)\cap\ZC_1$ is the line $\{(t, s\rho_tR\eta_1), s\in\R\}$,
  any point $(t, v)\in G(\theta)\setminus\ZC_1$ crosses between the two connected components of $G(\theta)\setminus\ZC_1$. By the invariance $\varphi_s(\AC)\subset\AC$ we get 

		$$\AC\setminus\ZC_{\XC_1}\neq\emptyset\hspace{.5cm}\implies\hspace{.5cm}\AC\cap (G(\theta)\setminus\ZC_1)\neq\emptyset\hspace{.5cm}\implies\hspace{.5cm}G(\theta)\setminus\ZC_1\subset\AC.$$
		From that, we conclude that $\partial_1f|_{G(\theta)\setminus\ZC_1}\equiv 0$ and by the continuity of $f$ we get actually 
		$$f(t, v)=f(0, v)\hspace{.5cm}\implies\hspace{.5cm}\partial_1f\equiv 0.$$
		However, since $A_10\mu R$ with $\mu\neq 0$, the relation
		$$\partial_2f(\varphi^1_s)=\rme^{-sA^{\intercal}}\partial_2f\hspace{.5cm}\implies\hspace{.5cm}\partial_2f(0, 0)=0,$$ 
		which together with $\partial_1f\equiv 0$  contradicts the fact that $\psi$ is a diffeomorphism. 	
	\end{enumerate}

    Therefore, in both cases, we conclude that $\AC\setminus\ZC_{\XC_1}=\emptyset$ or equivalently $\AC\subset\ZC_{\XC_1}$. Therefore, 
    $$\partial_2f|_{G(\theta)\setminus\ZC_{\XC_1}}\equiv 0,$$
    which by continuity implies 
    $$f(t, v)=at, \hspace{.5cm}\mbox{ where }\hspace{.5cm}a=\sigma_2/\sigma_1\in\R^*,$$ implying that $\partial_2f\equiv 0$ and concluding the proof.
\end{proof}

\begin{proposition}
	If $\R^2=\mathrm{Im}A_1\oplus\R\xi_1$ then $\partial_2f\equiv 0$.
\end{proposition}

\begin{proof}
	Let us assume w.l.o.g. that $(1, 0)\in\Delta_1$. In this case, $\nu_1=c\eta_1$ and from equations (\ref{9}) we get that 
	$$\partial_1f=(\alpha_1-c\alpha_2)\sigma_2\hspace{.5cm}\implies\hspace{.5cm}\partial_1f\;\mbox{ is bounded on }G(\theta)\setminus\ZC_1.$$
	On the other hand, by our hypothesis $\dim\mathrm{Im}A_1=1$ and so we have the following possibilities:
	
	\begin{enumerate}
		\item {\bf $A_1$ has a pair of distinct eigenvalues:} Since $A_1\theta=\theta A_1$ we have that, on the canonical basis,
	$$A_1=\left(\begin{array}{cc}
	\beta & 0 \\ 0 & 0
	\end{array}\right)\;\;\;\;\mbox{ or }\;\;\;\;A_1=\left(\begin{array}{cc}
	0 & 0 \\ 0 & \beta
	\end{array}\right)\;\;\;\mbox{ and }\;\;\; \theta=\left(\begin{array}{cc}
	1 & 0 \\ 0 & \lambda
	\end{array}\right).$$
	Let us assume that the first case holds for $A_1$ since the analysis of the second case is analogous. By our hypothesis, we can write $\xi_1=a\mathbf{e}_1+b\mathbf{e}_2$ with $b\neq 0$, and hence 
	$$\rho_t\Lambda_{-s}^{A_1}\xi_1=a\rme^{t}\frac{1}{\beta}(\rme^{-s\beta}-1)\mathbf{e}_1-bs\rme^{t\lambda}\mathbf{e}_2.$$
    Also, from the second and third equations in (\ref{eqf}) we have that   
    $$\partial_1f=\partial_1f(\varphi_s^1)+\langle\partial_2f(\varphi_s^1), \rho_t\Lambda_s^{A_1}\xi_1\rangle_{\R^2}=\partial_1f(\varphi_s^1)-\langle\partial_2f, \rho_t\Lambda_{-s}^{A_1}\xi_1\rangle_{\R^2},$$
    implying that 
    $$\langle\partial_2f, \rho_t\Lambda_{-s}^{A_1}\xi_1\rangle_{\R^2}=\partial_1f(\varphi_s^1)-\partial_1f,$$
    is bounded for $s\in\R$. As a consequence, 
    $$\langle\partial_2f, \rho_t\Lambda_{-s}^{A_1}\xi_1\rangle_{\R^2}=a\rme^{t}\frac{1}{\beta}(\rme^{-s\beta}-1)\langle\partial_2f, \mathbf{e}_1\rangle_{\R^2}-bs\rme^{t\lambda}\langle\partial_2f, \mathbf{e}_2\rangle_{\R^2},$$
    is bounded and hence 
    $$a\langle\partial_2f, \mathbf{e}_1\rangle_{\R^2}=0\hspace{.5cm}\mbox{ and }\langle\partial_2f, \mathbf{e}_2\rangle_{\R^2}=0.$$
    If $a=0$ we get that $\xi_1\in\ker A_1=\R\cdot \mathbf{e}_2$. In particular, $\Lambda_t^{\theta}\xi_1\in\ker A_1$ and 
    $$\langle\partial_2f, \mathbf{e}_2\rangle_{\R^2}=0\;\;\;\implies \;\;\;\langle\partial_2f, \Lambda_t^{\theta}\xi\rangle_{\R^2}=0.$$
    Using the first equation in (\ref{eqf}) we get that 
    $$0=\langle\partial_2f, A_1v+\Lambda_t^{\theta}\xi\rangle_{\R^2}=\langle\partial_2f, A_1v\rangle_{\R^2},$$
    which gives us that 
    $$\forall (t, v)\in G(\theta)\setminus(\{0\}\times\ker A_1), \;\;\;\;\langle\partial_2f, \mathbf{e}_1\rangle_{\R^2}=0,$$
    and in particular $\partial_2f\equiv 0$ on $G(\theta)\setminus\ZC_1$. Therefore, $\partial_1f=\sigma_2/\sigma_1=:a$ on $G(\theta)\setminus\ZC_1$ implying, by the continuity of $f$ that $f(t, v)=at$ on $G(\theta)$ which implies the result.
    
    \item {\bf $A_1$ is nilpotent:} As in the previous case, the fact that $A_1$ and $\theta$ commutes gives us that, on the canonical basis,
    $$A_1=\left(\begin{array}{cc}
    	0 & \beta \\ 0 & 0
    \end{array}\right)\;\;\;\mbox{ and }\;\;\; \theta=\left(\begin{array}{cc}
    1 & \delta \\ 0 & 1
\end{array}\right),  \;\delta\in\{0, 1\}\;\;\;\;\mbox{ or }\;\;\;\;$$
$$A_1=\left(\begin{array}{cc}
    	0 & 0 \\ \beta & 0
    \end{array}\right)\;\;\;\mbox{ and }\;\;\; \theta=\left(\begin{array}{cc}
    	1 & 0 \\ 0 & 1
    \end{array}\right).$$
    Let us as previously treat the first case. By the rank two assumption we have that $\xi_1=a\mathbf{e}_1+b\mathbf{e}_2$ with $b\neq 0$. Therefore, 
    $$\rho_t\Lambda_{-s}^{A_1}\xi_1=\rme^t\left(b\frac{s^2}{2}-(a+b\delta t)s\right)\mathbf{e}_1-\rme^tbs\mathbf{e}_2,$$
    and, as previously, $\partial_1f$ bounded on $G(\theta)\setminus\ZC_1$ implies 
    $$\forall s\in\R, \;\;\;\;\;\langle\partial_2f, \rho_t\Lambda_{-s}^{A_1}\xi_1\rangle_{\R^2}\;\;\;\mbox{ is bounded.}$$
    Therefore,
    $$\langle\partial_2f, \rho_t\Lambda_{-s}^{A_1}\xi_1\rangle_{\R^2}=\rme^t\left(b\frac{s^2}{2}-(a+b\delta t)s\right)\langle\partial_2f, \mathbf{e}_1\rangle_{\R^2}-\rme^tbs\langle\partial_2f, \mathbf{e}_2\rangle_{\R^2}$$
    is bounded for $s\in\R$ implying that 
    $$\frac{\rme^tb}{2}\langle\partial_2f, \mathbf{e}_1\rangle_{\R^2}=0\;\;\;\mbox{ and }\;\;\;\rme^t\left((a+b\delta t)\langle\partial_2f, \mathbf{e}_1\rangle_{\R^2}+b\langle\partial_2f, \mathbf{e}_2\rangle_{\R^2}\right)=0.$$
    Since by assumption $b\neq 0$, the previous is equivalenty to
    $$\langle\partial_2f, \mathbf{e}_1\rangle_{\R^2}=\langle\partial_2f, \mathbf{e}_2\rangle_{\R^2}=0,$$
    showing that $\partial_2f\equiv 0$ on $G(\theta)\setminus\ZC_1$. As previously, by continuity $f(t, v)=at$ and hence $\partial_2f\equiv 0$ on $G(\theta)$.
    \end{enumerate}
\end{proof}

\subsection{Proof of Theorem \ref{main}}

By the previous section, if $\psi=(f, g)$ is an isometry between rank two ARSs $\Sigma_1$ and $\Sigma_2$, then $\partial_2f\equiv 0$. As a consequence, if 
$$X=(0, \eta)\in \Delta_1\cap\left(\{0\}\times\R^2\right), \;\;\;\mbox{ with }\;\;\;\|X\|_{\Sigma_1, (0, 0)}=1,$$
we have that 
$$(d\psi)_{(t, v)}X^L(t, v)=\left(\begin{array}{cc}
\partial_1f(t, v) & 0 \\ \partial_1g(t, v) & \partial_2g(t, v)
\end{array}\right)\left(\begin{array}{c}
0 \\ \rho_t\eta
\end{array}\right)=(0, \partial_2g(t, v)\rho_t\eta),$$
showing that 
$$(d\psi)_{(t, v)}X^L(t, v)\in \Delta_2(\psi(t, v))\cap\fn(\theta).$$
On other hand, 
$$(d\psi)_{(0, 0)}X\in\Delta_2(0, 0)\;\;\;\implies\;\;\;(dL_{\psi(t, v)})_{(0, 0)}(d\psi)_{(0, 0)}X\in\Delta_2(\psi(t, v))\cap\fn(\theta).$$
Since $\psi$ is an isometry and the metrics on $\Delta_1$ and $\Delta_2$ are left-invariant, we get that
$$\|(d\psi)_{(t, v)}X^L(t, v)\|_{\Sigma_2, (\psi(t, v))}=\|X^L(t, v)\|_{\Sigma_1, (t, v)}=\|X\|_{\Sigma_1, (0, 0)}=1,$$
and
$$\|(dL_{\psi(t, v)})_{(0, 0)}(d\psi)_{(0, 0)}X\|_{\Sigma_2, (\psi(t, v))}=\|(d\psi)_{(0, 0)}X\|_{\Sigma_2, (0, 0)}=\|X\|_{\Sigma_1, (0, 0)}=1.$$

The fact that $\dim(\Delta_2(\psi(t, v)))\cap\left(\{0\}\times\R^2\right)=1$ implies
$$(d\psi)_{(t, v)}X^L(t, v)=\pm(dL_{\psi(t, v)})_{(0, 0)}(d\psi)_{(0, 0)}X,$$
where the sign is constant on any connected component of $G(\theta)\setminus\ZC_1$. An analogous analysis as the one done in the proof of Step 1 of the Fundamental Lemma implies that the sign is constant on the whole $G(\theta)$ and consequently 
$$(d\psi)_{(t, v)}X^L(t, v)=(dL_{\psi(t, v)})_{(0, 0)}(d\psi)_{(0, 0)}X.$$
By the Fundamental Lemma (\ref{fundamentallemma}) we conclude that $\psi$ is an automorphism of $G(\theta)$, concluding the proof.

\subsection{A classification result for rank two ARSs}	
	
	The results in the previous sections allow us to obtain a classification result for rank two simple ARSs.
	
	For any $\sigma\in\R^+$ we consider the subsets of $\fg(\theta)$ given by
	$$\alpha_1=\{(1, 0), (\sigma, \mathbf{e}_1)\}, \;\;\alpha_2=\{(1, 0), (\sigma, \mathbf{e}_2)\}\;\;\mbox{ and }\;\;\alpha_3=\{(1, 0), (\sigma, \mathbf{e}_1+\mathbf{e}_2)\}.$$
	Define the simple ARSs of rank two $\Sigma^i_{\XC, \sigma}=\{\XC, \Delta^L_{i, \sigma}\}$ on $G(\theta)$, where 
	\begin{enumerate}
		\item $\XC$ is a rank two linear vector field on $G(\theta)$;
		\item $\alpha_i$ is an orthonormal basis of $\Delta_{i, \sigma}$.
	\end{enumerate}
	
   Denote by $\EC_{\theta}$ the set of all rank two simple ARS on $G(\theta)$ and consider the sets 
   $$\EC^i_{\theta}=\{\Sigma\in \EC_{\theta}; \;\mathrm{Iso}(\Sigma; \Sigma^i_{\XC, \sigma})_0\neq\emptyset\;\mbox{ for some }\Sigma^i_{\XC, \sigma}\},$$
   that is, $\EC^i_{\theta}$ is the set of rank two ARSs on $G(\theta)$ that are isometric to some of the ARSs $\Sigma^i_{\XC, \sigma}$.

 \begin{theorem}
	 	Up to a reescaling, it holds that
	 	\begin{itemize}
	 		\item[(i)] $\EC_{\theta}=\EC^1_{\theta}$ if $\theta\in\left\{\left(\begin{array}{cc}
	 			1 & 0\\0 & 1
	 		\end{array}\right), \; \left(\begin{array}{cc}
	 		\gamma & -1\\1 & \gamma
 		\end{array}\right), \;\gamma\in\R\right\}$;
 	
 	\item[(ii)]  $\EC_{\theta}=\EC^1_{\theta}\,\dot{\cup}\,\EC^3_{\theta}$ if $\theta\in\left\{\left(\begin{array}{cc}
 		1 & 1\\0 & 1
 	\end{array}\right), \; \left(\begin{array}{cc}
 		1 & 0\\0 & -1
 	\end{array}\right)\right\}$;
 
 \item[(iii)]  $\EC_{\theta}=\EC^1_{\theta}\,\dot{\cup}\,\EC^2_{\theta}\,\dot{\cup}\,\EC^3_{\theta}$ if $\theta\in\left\{\left(\begin{array}{cc}
 	1 & 0\\0 & \gamma
 \end{array}\right), \gamma\in(-1, 1)\right\}$.
	 	\end{itemize}
  \end{theorem}
	 \begin{proof}
	 	Since, by Theorem \ref{main}, the only isometries between rank two ARS are automorphisms, we only have to show that any given ARS $\Sigma$ is isometric to some ARS in $\EC^i_{\theta}$ for $i=1, 2, 3$ and that for $i\neq j$ we have that $\EC^i_{\theta}\cap \EC^j_{\theta}=\emptyset$ if $\theta$ is in the cases (ii) or (iii).
	 	
	    Let us consider $\Sigma=\{\XC, \Delta^L\}$ be a rank two ARS. By Proposition \ref{difeo} the ARS $\Sigma$ is isometric to an ARS whose distribution contains $(1, 0)$. Consequently, we can assume w.l.o.g. that $(1, 0)\in\Delta$. By reescaling the norm on $\Delta$ if necessary we can assume that $\|(1, 0)\|_{\Sigma, (0, 0)}=1$. Choose $(\sigma, \eta)\in \Delta$ such that $\sigma>0$ and $\{(1, 0), (\sigma, \eta)\}$ is an orthonormal basis of $\Delta$. Note that, in this case  
	 	$$(0, \eta )=-\sigma (1, 0)+ (\sigma, \eta)\hspace{.5cm}\implies\hspace{.5cm} l_{\Delta}=\R \cdot\eta.$$
	 	Write $\eta=(x, y)$ and consider the following cases:
	 	
	 	\bigskip
	 	
	 	{\bf Case 1:} $\theta\in\left\{\left(\begin{array}{cc}
	 		1 & 0\\ 0 & 1
	 	\end{array}\right), \;\left(\begin{array}{cc}
	 		\gamma & -1\\ 1 & \gamma
	 	\end{array}\right), \;\gamma\in\R\right\}$.
	 	
	 	In this case, by considering $P=\left(\begin{array}{cc}
	 		x & -y\\ y & x
	 	\end{array}\right)$ we have that  
	 	$$P\theta=\theta P, \hspace{.5cm}\det P=x^2+y^2\neq 0\hspace{.5cm}\mbox{ and }\hspace{.5cm} P\mathbf{e}_1=\eta.$$
	 	The automorphism $\phi(t, v)=(t, Pv)$ is an isometry between $\Sigma^1_{\XC_{\phi}, \sigma}$ and $\Sigma$. 
	 	
	 	\bigskip
	 	
	 	{\bf Case 2:} $\theta=\left(\begin{array}{cc}
	 		1 & 1\\ 0 & 1
	 	\end{array}\right)$.
	 	
	 	In this case, by considering  
	 	$$P_1=\left(\begin{array}{cc}
	 		x & 0\\ 0 & x
	 	\end{array}\right)\;\;\;	\mbox{ if }\;\;\;y=0\;\;\;\mbox{ or }\;\;\; P_3=\left(\begin{array}{cc}
	 		y & x-y\\ 0 & y
	 	\end{array}\right)\;\mbox{ if }\;y\neq 0,$$
 	     we get that $P_i\theta=\theta P_i$, $\det P_i\neq 0$ and 
	 	$$P_1\mathbf{e}_1=\eta\;\;\mbox{ if }\;y=0\;\;\mbox{ and }\;\;P_3(\mathbf{e}_1+\mathbf{e}_2)=\eta\;\;\mbox{ if }\;y\neq 0.$$
	 	Therefore, if $y=0$ the automorphism $\phi_1(t, v)=(t, P_1v)$ is an isometry between $\Sigma^1_{\XC_{\phi_1}, \sigma}$ and $\Sigma$, and if $y\neq 0$ the automorphism $\phi_3(t, v)=(t, P_3v)$ is an isometry between $\Sigma^3_{\XC_{\phi_3}, \sigma}$ and $\Sigma$.
		 
		 \bigskip
	 	
	 	{\bf Case 3:} $\theta=\left(\begin{array}{cc}
	 		1 & 0\\ 0 & -1
	 	\end{array}\right)$.
	 	
	 	In this case, we consider 
	 	$$P_1=\left(\begin{array}{cc}
	 		x & 0\\ 0 & x
	 	\end{array}\right)\mbox{ if }y=0, \;P_2=\left(\begin{array}{cc}
	 		0 & y\\ y & 0
	 	\end{array}\right)\mbox{ if }x=0\mbox{ or } P_3=\left(\begin{array}{cc}
	 		x & 0\\ 0 & y
	 	\end{array}\right)\;\mbox{ if }\;xy\neq 0.$$ 
	 	It holds that $P_i\theta=\theta P_i$ for $i=1 \mbox{ or } 3$ and $P_2\theta=-\theta P_2$, $\det P_i\neq 0$ and 
	 	$$P_1\mathbf{e}_1=\eta\;\;\mbox{ if }\;y=0, \;\;P_2\mathbf{e}_1=\eta\;\;\mbox{ if }\;y=0\;\;\mbox{ and }\;\;P_3(\mathbf{e}_1+\mathbf{e}_2)=\eta\;\;\mbox{ if }\;xy\neq 0.$$
	 	As in the previous case, if $y=0$ the automorphism $\phi_1(t, v)=(t, P_1v)$ ({\it resp. if $x=0$ the automorphism $\phi_2(t, v)=(-t, P_2v)$}) is an isometry between $\Sigma^1_{\XC_{\phi_1}, \sigma}$ ({\it resp. $\Sigma^1_{\XC_{\phi_2}, \sigma}$}) and $\Sigma$,  and $\phi_3(t, v)=(t, P_3v)$  is an isometry between $\Sigma^3_{\XC_{\phi_3}, \sigma}$ and $\Sigma$ if $xy\neq 0$.
	 	
	 	{\bf Case 4:} $\theta=\left(\begin{array}{cc}
	 		1 & 0\\ 0 & \lambda
	 	\end{array}\right), \;\;\lambda\in (1, -1)$.
	 	
	 	In this case, we consider 
	 	$$P_1=\left(\begin{array}{cc}
	 		x & 0\\ 0 & x
	 	\end{array}\right)\hspace{.3cm}\mbox{ if }\hspace{.3cm}y=0,\hspace{.3cm} P_2=\left(\begin{array}{cc}
	 		y & 0\\ 0 & y
	 	\end{array}\right)\hspace{.3cm}\mbox{ if }\hspace{.3cm}x=0\hspace{.3cm}\mbox{ or } \hspace{.3cm}P_3=\left(\begin{array}{cc}
	 		x & 0\\ 0 & y
	 	\end{array}\right)\hspace{.3cm}\mbox{ if }\hspace{.3cm}xy\neq 0.$$ 
	 	Again by construction $P_i\theta=\theta P_i$, $\det P_i\neq 0$ and 
	 	$$P_1\mathbf{e}_1=\eta\;\;\mbox{ if }\;y=0, \;\;P_2\mathbf{e}_2=\eta\;\;\mbox{ if }\;x=0\;\;\mbox{ and }\;\;P_3(\mathbf{e}_1+\mathbf{e}_2)=\eta\;\;\mbox{ if }\;xy\neq 0,$$
	 	and the respectively automorphisms $\phi_i(t, v)=(t, P_iv)$ are isometries between $\Sigma_{\XC_{\phi_i},, \sigma}$ and $\Sigma$, for $i=1, 2, 3$.
	 	
	 	Since the previous cases cover all the possibilities, we have that $\EC$ is in fact decomposed by the classes $\EC_{\theta}^i$ as given in items (i), (ii) and (iii). The only thing that remains to show is that, in cases (ii) and (iii) we have that $\EC^i_{\theta}\cap \EC^j_{\theta}=\emptyset$ for $i\neq j$. 
	 	
	 	Since both cases are analogous, let us show case (ii). In this case, if $\EC^1_{\theta}\cap\EC^3_{\theta}\neq\emptyset$, there exists rank two linear vector fields $\XC_1, \XC_3$ and positive real numbers $\sigma_1, \sigma_2$ such that $\Sigma^1_{\XC_1, \sigma_1}$ and $\Sigma^3_{\XC_3, \sigma_3}$ are isometrics. However, since $\Sigma^1_{\XC_1, \sigma_1}$ and $\Sigma^3_{\XC_3, \sigma_3}$ are, by definition, rank two ARSs, Theorem \ref{main} implies that $$\mathrm{Iso}\left(\Sigma^1_{\XC_1, \sigma_1}; \Sigma^3_{\XC_3, \sigma_3}\right)_0\subset\mathrm{Aut}(G(\theta)).$$
	 	Therefore, any $\psi\in \mathrm{Iso}\left(\Sigma^1_{\XC_1, \sigma_1}; \Sigma^3_{\XC_3, \sigma_3}\right)_0$ satisfies  
	 	$$(d\psi)_{(0, 0)}=\left(\begin{array}{cc}
	 		\varepsilon & 0\\ \eta & P
	 	\end{array}\right), \;\;\mbox{ with }\;\; P\theta=\varepsilon\theta P,$$
 	    which implies, in particular, that 
 	    $$(d\psi)_{(0, 0)}\left(\Delta_{1, \sigma}\cap\fn(\theta)\right)=\Delta_{3, \sigma}\cap\fn(\theta)\hspace{.5cm}\mbox{ and hence }\hspace{.5cm}P\mathbf{e}_1\in\R\cdot(\mathbf{e}_1+\mathbf{e}_2).$$
 	    However, by the hypothesis on $\theta$, the subspace $\R\cdot \mathbf{e}_1$ is a one dimensional eigenspace of $\theta$. Since $P\theta=\varepsilon \theta P$ we get that
 	    $$P\mathbf{e}_1\in\R \cdot \mathbf{e}_1 \;\mbox{ if }\;\varepsilon=1\hspace{.5cm}\mbox{ and }\hspace{.5cm}P\mathbf{e}_1\in\R\cdot \mathbf{e}_2 \;\mbox{ if }\;\varepsilon=-1,$$
 	    which contradicts $P\mathbf{e}_1\in\R\cdot(\mathbf{e}_1+\mathbf{e}_2)$. Therefore, $\EC^1_{\theta}\cap\EC^3_{\theta}=\emptyset$ as stated.	 	
	 \end{proof}

	 \begin{remark}
	 	In the notation of the previous result, let us note that
	 	$$(0, \eta)=-\sigma(1, 0)+(\sigma, \eta)\hspace{.5cm}\implies\hspace{.5cm}\|(0, \eta)\|_{\Sigma, (0, 0)}=1+\sigma^2.$$
	 	In particular, $\|(0, \eta)\|_{\Sigma, (0, 0)}=1$ if and only if $\sigma=0$. Consequently, the metric on $\Delta$ is Euclidean if and only if $\sigma=0$.
	 \end{remark}

\end{document}